\documentclass[final, reqno]{amsart}
\usepackage{amsmath}
\usepackage{amsfonts}
\usepackage{amsthm}
\numberwithin{equation}{section}
\title{Local Exact Controllability to Stationary Solutions of a Semilinear Parabolic Equation}
\author{Marius Beceanu}
\thanks{Department of Mathematics, University of Chicago, 5734 S.\ University Ave., Chicago, IL 60637 ({\tt mbeceanu@uchicago.edu})}

\newtheorem{lemma}{Lemma}[section]
\newtheorem{proposition}[lemma]{Proposition}
\newtheorem{corollary}[lemma]{Corollary}
\newtheorem{observation}[lemma]{Remark}
\newtheorem{theorem}[lemma]{Theorem}
\newcommand{\set}{\mathbb}

\newcommand{\dl}{\nabla}

\newcommand{\ye}{y_{\epsilon}}
\newcommand{\ue}{u_{\epsilon}}
\newcommand{\pe}{p_{\epsilon}}
\newcommand{\LL}{\mathcal L}
\newcommand{\be}{\begin{equation}}
\newcommand{\ee}{\end{equation}}
\newcommand{\lb}{\label}
\newcommand{\supp}{supp}
\renewcommand{\div}{div}
\renewcommand{\upsilon}{R}
\newcommand{\ov}{\overline}
\newcommand{\ds}{\displaystyle}
\newcommand{\dd}{{\,d}}
\usepackage [frenchlinks=true]{hyperref}

\begin{document}
\maketitle
\begin{abstract}
This paper establishes the local exact controllability of the quasilinear porous media equation with Dirichlet boundary conditions.

Consider the equation
\begin{equation*}
\begin{array}{l}
y_t - \Delta a(y) = m u + f \text{ on }Q\\
y(0) = y_0,\ y \mid_{\Sigma} = 0
\end{array}
\end{equation*}
on the $n + 1$-dimensional cylinder $Q = \Omega \times (0, T)$ with lateral boundary $\Sigma  = \partial \Omega \times (0, T)$. The exact controllability in finite time is proved when $\|y_0 - y_s\|_{W^{1, n}_0(\Omega) \cap C(\ov \Omega)}$ is sufficiently small, $n > 1$, for every stationary solution $y_s$ such that $a(y_s) \in W^{2, q}(\Omega)$, where $q>n$. It is assumed that $\Omega$ is a bounded open set with $C^2$ boundary and that $a \in C^2(\set R)$, $a'>0$.
\end{abstract}



\pagestyle{myheadings}
\thispagestyle{plain}
\markboth{M. BECEANU}{LOCAL EXACT CONTROLLABILITY OF THE NONLINEAR DIFFUSION EQUATION}

\section{Introduction}
\subsection{Statement and Main Result}
Consider a fixed cylinder $Q = Q_T = \Omega \times (0, T)$, where $\Omega$ is a bounded domain with $C^2$ boundary in $\set R^n$, $n>1$. This paper aims to establish the local exact controllability of the quasilinear porous media equation with Dirichlet boundary conditions, namely
\begin{equation}
\begin{array}{l}
y_t - \Delta a(y) = mu + f \text{ on } Q\\
y(0) = y_0,\ y \mid_{\Sigma = \partial \Omega \times (0,T)} = 0.
\end{array}
\label{dif}
\end{equation}
Here $a: \set R \rightarrow \set R$ is a function that only matters up to an additive constant. We will assume that $a \in C^2(\set R)$, $a(0)=0$, and $a'(y) >0$ $\forall y \in \set R$. Moreover, let $f$ be constant in time, $f \in L^q(\Omega)$, $q>n$.

Here $m$ is the characteristic function of an open subset $\omega \subset \subset \Omega$. We also denote $Q_{\omega} = \omega \times (0, T)$ and $\|\cdot\|_q = \|\cdot\|_{L^q(\Omega)}$.

Consider $y_s$ to be a stationary solution of (\ref{dif}), such that $a(y_s) \in W^{2, q}(\Omega)$ and therefore $y_s \in W^{2, q}(\Omega)$:
\begin{equation}
\begin{array}{l}
-\Delta a(y_s) = f,\ y_s \mid_{\partial \Omega} = 0.
\end{array}
\label{stationar}
\end{equation}

The equation (\ref{dif}) is said to be exactly controllable to $y_s$ if there is $u \in L^2(Q)$, called the control, such that by inserting it in (\ref{dif}) one obtains a solution $y$ such that $y(T) \equiv y_s$. If equation (\ref{dif}) is exactly controllable to $y_s$ for all $y_0$ in some neighborhood of $y_s$, the equation is called locally exactly controllable.

This paper proves the local exact controllability of equation (\ref{dif}).

In the statement of the main result we use several notations that depend on the following lemma, due to~\cite{fursikov}:
\begin{lemma}
There exists a function $\psi \in C^2(\overline{\Omega})$ such that $\psi(x) > 0$ $\forall x \in \Omega$, $\psi(x) = 0$ on $\partial \Omega$, and $|\dl \psi(x)| > 0$ $\forall x \in \Omega \setminus \omega_0$, where $\omega_0 \subset \subset \omega$.
\label{lema}
\end{lemma}

Making use of $\psi$ define, for $\lambda > 0$, functions $\alpha$ and $\phi: Q \rightarrow \set R$ by
\begin{equation}
\alpha(x, t) = \frac{e^{\lambda \psi(x)} - e^{2 \lambda \|\psi\|_{C(\overline{\Omega})}}} {t (T - t)},
\phi(x, t) = \frac{e^{\lambda \psi(x)}} {t (T - t)}.
\lb{def_ap}
\end{equation}

Also consider the auxiliary functions
\begin{equation}
\alpha_0(t) = \frac {1 - e^{2\lambda \|\psi\|_{C(\overline{\Omega})}}}{t (T - t)}, \phi_0(t) = \frac 1 {t (T - t)}.
\end{equation}
Note that that $\alpha_0 \leq \alpha \leq \alpha_0 (1 - \eta(\lambda))$, $\phi_0 \leq \phi \leq \eta^{-1}(\lambda) \phi_0$, where $\eta(\lambda) =~e^{-\lambda \|\psi\|_{C(\overline{\Omega})}}$.

Our main result is the following:
\begin{theorem}
Assume that $a \in C^2(\set R)$ and $a'(y)>c$ on $(-\epsilon, \epsilon)$. Consider $y_s$ a stationary solution of (\ref{dif}), with $a(y_s) \in W^{2, q}(\Omega)$, and $T>0$. Then there is $C(a, y_s, T)$ such that for all $y_0$ with
\be
\|y_0 - y_s\|_{W^{1, n}_0(\Omega)} + \|y_0-y_s\|_{C(\ov \Omega)} \leq C(a, y_s, T),
\ee
there exists a control $u$ that produces a solution $y$ of (\ref{dif}) with $y_0$ initial data and $y(T) = y_s$. Furthermore, for some $q>n$ and any $\lambda \geq \lambda_0(a, y_s, T)$, $s \geq s_0(\lambda)$,
\begin{equation}
\|e^{-s \alpha_0}u\|_{L^{\infty}(Q)} \leq C_{s, \lambda} C_{a, y_s, T} \|y_0 - y_s\|_2^{2/q}.
\label{control}
\end{equation}
The controlled solution $y$ has the property that
\begin{equation}
\|t(y-y_s)\|_{W^{1, \infty}(0, T; L^q(\Omega)) \cap L^{\infty}(0, T; W^{1, \infty}(\Omega))} \leq C_{a, y_s, T},
\ee
\be
\|e^{-s\alpha_0}(y-y_s)\|_{C([T/2, T]; L^{q}(\Omega))} \leq C_{a, y_s, T} \|y_0 - y_s\|_2^{2/{q}},
\label{control2}
\end{equation}
as well as
\begin{equation}
\|y-y_s\|_{L^{\infty}(Q)} \leq C_{a, y_s, T}(\|y_0-y_s\|_2^{2/q} + \|y_0-y_s\|_{\infty}).
\label{estinfinit}
\end{equation}
The constants depend on $a$ and $y_s$ by means of $\epsilon$, $\|a'\|_{C^1[-\epsilon, \epsilon]} + \|1/a'\|_{C[-\epsilon, \epsilon]}$ and $\|y_s\|_{W^{2, q}(\Omega)}$, respectively.
\label{2nd}
\end{theorem}
\begin{observation}
The control $u$ and the controlled solution $y$ depend continuously on the initial data $y_0$ as a function from $C(\ov \Omega)$ to $L^{\infty}(Q) \times C(\ov Q)$.
\end{observation}
\begin{observation}
Let $T>0$. Given an arbitrary (not necessarily stationary) solution $y_e$ of equation (\ref{dif}), the equation is said to be locally exactly controllable to $y_e$ in time $T$ such that for any $y_0$ in a neighborhood of $y_e(0)$ there exists a control $u$ depending continuously on $y_0$ such that, by plugging $u$ in (\ref{dif}), one obtains a solution $y$ such that $y(T) \equiv y_e(T)$. The method of proof given in this paper can be used, with minimal modifications, to prove the local exact controllability of equation (\ref{dif}) to any $L^{\infty}(0, T; W^{2, q}(\Omega)) \cap W^{1, \infty}(0, T; L^q(\Omega))$ solution $y_e$, for $q>n$.
\end{observation}

\subsection{History of the Problem}
This paper is based on the method introduced by Imanuvilov and Fursikov in 1992-93 over the course of several articles, as for example \cite{fursikov3}, and presented in a systematical manner in \cite{fursikov}. The method allowed the proof of observability (and, by duality, controllability) for solutions of parabolic equations over arbitrary domains, with nonconstant coefficients, as opposed to previous controllability results, such as \cite{russell} and \cite{henry}.

In \cite{fursikovrusa}, Imanuvilov proved the exact controllability of the semilinear parabolic equation
\begin{equation}
\left\{\begin{array}{ll}
\LL y + f(t, x, y) = g + u \text{ on }Q = [0, T] \times \Omega\\
y\mid_{[0, T]\times\partial\Omega} = 0,\ y(0, x) = v_0(x), y(T, x) = v_1(x).
\end{array}\right.
\label{eq1}
\end{equation}
In this formulation, $\Omega$ is a bounded domain in $\set R^n$, $\LL$ is a parabolic operator and the control $u$ is supported on $\omega \subset \subset \Omega$. The initial state $v_0$ is taken in $H^1_0(\Omega)$ and the nonlinear term $f$ satisfies $f(t, x, 0) = 0$ and is globally Lipschitz with respect to $y$. Imanuvilov proves the exact controllability of \eqref{eq1} to states $v_1$ for a class of states containing $0$, together with the estimate
\begin{multline}
\|y\|_{L^2([0, T]; H^2(\Omega))} + \|y_t\|_{L^2(Q)} + \|e^{s \eta} y\|_{L^2(Q)} + \|u\|_{L^2(Q)} \leq\\
C(\|v_0\|_{H^1_0(\Omega)} + \|e^{s \eta} g\|_{L^2(Q)}),
\end{multline}
where $\eta$ is an auxiliary function, defined by the author, comparable to $(T-t)^{-2}$.

The main elements of Imanuvilov's proof are a Carleman-type inequality for the linear problem and a fixed-point principle, by means of which the controllability result is extended to the semilinear problem.

This method was afterwards used by several other authors to prove more general controllability results. Thus, Fernandez-Cara, Zuazua \cite{fernandez} proved the global exact controllability of the semilinear equation $y_t-\Delta y + f(y)=mu$, with $|f(s)|<ks\log |s|$, as well as a negative result -- the equation cannot be controlled even for $f$ of moderate growth, $|f(s)| \approx s \log^{2+\epsilon}(1+|s|)$. The same authors also proved results concerning nonlinearities involving the gradient and on unbounded domains.

Barbu \cite{barbu2} showed that the preceding result can be extended to the case when $|f(x, t, y)| \leq k_1 |y| (1+o(|y|)(1+\log |y|)^{3/2})$.

In the current paper and in \cite{sine}, controllability is proved for the quasilinear parabolic equation (\ref{dif}). This is not the most general quasilinear equation, insofar as the coefficient of the main term is a function of $y$, but not of $\dl y$. Nevertheless, it has physical relevance, since it describes the diffusion of fluids within porous media with a nonlinear diffusion coefficient. Whereas \cite{sine} deals with the one-dimensional problem, here we treat the case when the dimension is at least two.

This paper is fundamentally based on the method of proof of \cite{fursikov}. The computations (especially those in the proof of the Carleman inequality) closely follow those of \cite{barbu}.


\section{Overview of the Proof} Firstly, we linearize equation (\ref{dif}) by introducing an auxiliary function $z$:
\begin{equation}
\begin{array}{l}
y_t - \div (a'(z) \dl y) = m u + f \text{ on } Q\\
y \mid_{\Sigma} = 0,\ y(0) = y_0.
\end{array}
\label{y}
\end{equation}
Indeed, by making $y = z$, one retrieves the original equation (\ref{dif}).

Denote $b = a'(z)$. The second step is to prove a Carleman-type inequality concerning the backward dual (\ref{p}) to the linear equation (\ref{y}):
\begin{equation}
\begin{array}{l}
p_t + \div (b \dl p) = g\\
p \mid_{\Sigma} = 0,\ p \mid_{t = T} = p_T.
\end{array}
\label{p}
\end{equation}

The inequality that needs proof is as follows:
\begin{proposition}
Assume that $0 < \mu \leq b \leq M < \infty$, where the constants $\mu$ and $M$ are fixed. Whenever $\|b_t\|_{L^{\infty}(0, T; L^n(\Omega))} + \|\dl b\|_{L^{\infty}(Q)} \leq \zeta < \infty$ and for every $\lambda \geq \lambda_0(\zeta) = C (1 + \zeta^2)$, $s \geq s_0(\lambda)$, any solution $p$ of (\ref{p}) satisfies the bound
\begin{multline}
\int_Q {e^{2s\alpha}(s^3 \lambda^3 \phi^3 p^2 + s \lambda \phi (\dl p)^2) dx dt} \leq \\
\leq C (1 + \zeta) \int_{Q_{\omega}} {e^{2s\alpha} s^3 \lambda^3 \phi^3 p^2 dx dt} + \int_Q {e^{2s\alpha} g^2 dx dt},
\label{Carleman}
\end{multline}
where $C$ is a constant independent of $p$, $g$, $s$, $\lambda$, and $\zeta$, but which may depend on $T$, $\Omega$, and $\omega$, as well as $\mu$ and $M$.
\label{car}
\end{proposition}

Recall that $Q_{\omega} = \omega \times (0,T)$.

In particular, the assumptions of Proposition \ref{car} imply that $b$ is H\"older continuous and that its modulus of continuity depends only on $\zeta$. Since $b$ is given by $b=a'(z)$, (\ref{Carleman}) holds uniformly for all $z$ in a bounded set in $L^{\infty}(0, T; W^{1, \infty}(\Omega)) \cap W^{1, \infty}(0, T; L^n(\Omega))$. We also remark that the independence of $C$ in regard to $s$ and $\lambda$ will be used later in the proof.

This statement of the Carleman inequality contains a minor but useful improvement over, for instance, the one given in \cite{yamamoto}. Namely, we do not require that all derivatives of $b$ should be in $L^{\infty}(Q)$.

The observability of the dual equation, established by (\ref{Carleman}), is equivalent to the controllability of the linear equation (\ref{y}). Under some specific conditions on the auxiliary function $z$, we show that the linear equation (\ref{y}) admits a solution $y$ which fulfills the same conditions. In order to pass from the linear equation (\ref{y}) to the quasilinear equation (\ref{dif}), we then apply the Schauder fixed point theorem in the Banach space $C(\ov Q)$.


Fix $q \in (n, \infty)$ and take $\ds N = \left\lceil\frac {\log q - \log 2}{\log{n} - \log{(n-2)}}\right\rceil+1$ for $n>2$ and $N=1$ for $n=2$. Consider the sequence
\be
(q_i)_{0 \leq i \leq N},\ q_0 = 2,\ q_i = 2 \Big(\frac {n}{n-2}\Big)^{i-1},\ 0<i<N,\ q_N = q.
\lb{def_q}
\ee

Define the compact set $B \subset C(\ov Q)$, depending on the parameters $\zeta_i$ and $\zeta$, as follows:
\begin{equation}
\begin{split}
B = \{y \in C(\ov Q) \mid& \|e^{s\alpha} \phi^{-i} (y-Y)\|_{L^{q_i}(0, T/2; L^{q_{i+1}}(\Omega))} + \\
& + \|e^{s\alpha} \phi^{-i} (y-y_s)\|_{L^{q_i}(T/2, T; L^{q_{i+1}}(\Omega))} \leq \zeta_i,\ 0 \leq i \leq N,\\
&y(0)=y_0,\ y \mid_{\Sigma} = 0,\ \|y_t\|_{C([0, T]; L^q(\Omega))} \leq \zeta,\ \|\dl(y-y_s)\|_{L^{\infty}(Q)} \leq \zeta\},
\label{B}
\end{split}
\end{equation}
where $Y$ is the uncontrolled solution of the quasilinear equation (\ref{omogen}) obtained by taking $u=0$ in (\ref{dif}):
\begin{equation}
Y_t - \Delta a(Y) = f,\ Y(0)=y_0,\ Y \mid_{\Sigma} = 0.
\label{omogen}
\end{equation}
$B$ is defined so that Carleman's inequality Proposition \ref{car} holds uniformly if $b=a'(z)$ and $z \in B$; thus, one can close the fixed point argument in $B$.

Note that, for $y \in B$, $y - Y$ vanishes exponentially at time zero and $y-y_s$ vanishes exponentially at time $T$. In particular, the latter implies that $y(T) = y_s$, but both properties will matter in the sequel.

Define the mapping $F: B \rightarrow C(\ov Q)$ by assigning to each $z$ the minimizer $F(z) = y$ of the functional
\begin{equation}
\begin{aligned}
Q_z(y, u) = &\int_{Q_{\omega}} {e^{-2s \alpha} \phi^{-3} u^2 dx dt} + \\
&+ \int_Q {  e^{-2s\alpha} (\chi_{[0, T/2]} (y(x, t) - Y(x, t))^2 + \chi_{[T/2, T]} (y(x, t) - y_s(x))^2) dx dt},
\label{F}
\end{aligned}
\end{equation}
where $y$ and $u$ satisfy (\ref{y}), with $z$ as the auxiliary function.

To $F$ we apply Schauder's Theorem, which states that if $B$ is a closed and convex set in a Banach space $X$ and $F:B \rightarrow B$ is a continuous map such that $F(B)$ is compact, then there exists $x \in B$ such that $F(x)=x$.

We also use the fact that if $X$ is a normed space, $F:B \to X$ has a closed graph, and $F(B)$ is compact, then $F$ is continuous.

Therefore, we next show the existence of a fixed point $y$ of $F$ in $B$ by means of Schauder's Theorem, for a suitable choice of the parameters $\zeta_i$ and $\zeta$. Due to the rapid growth of $e^{-2s\alpha}$, the difference $y -y_s$ will vanish at time $T$ as desired for $y \in B$. This result is contained in the following more technical statement:

\begin{proposition}
Assume that $a'$, $a''$, and $1/a'$ are bounded on $\set R$. Fix $q \in [n, \infty)$. Assume that $y_0 \mid_{\Sigma} = 0$ and $\|y_0-y_s\|_q + \|\Delta a(y_0) + f\|_q \leq C(a, y_s)$. Then, the quasilinear equation (\ref{dif}) is controllable to the stationary solution $y_s$ and for any $T$ there exists $\lambda_0$ such that for any $\lambda \geq \lambda_0$,  $s \geq s_1(\lambda)$, $\delta>0$, and $0\leq i< N$ there exists a control $u$ satisfying
\begin{equation}
\begin{array}{c}
\displaystyle{
\|e^{-s(1 - 2\eta(\lambda) - \delta) \alpha_0}u\|_{L^{\infty}(Q)} + \|e^{-s(1 - 2\eta(\lambda) - \delta) \alpha_0}\phi_0^{-1}u_t\|_{L^q(Q)} \leq }\\
\displaystyle{
\leq C_{s, \lambda, \delta} C_{y_s, a} \|y_0 - y_s\|_q, }\\
\|e^{-s\alpha} \phi_0^{1-N} u\|_{C([0, T]; L^{q_i}(\Omega)) \cap L^{q_i}(0, T; L^{q_{i+1}}(\Omega))} \leq C_{y_s, a} \|y_0 - y_s\|_{q_i}.
\end{array}
\label{est1}
\end{equation}
The controlled solution $y$ has the property that
\begin{equation}
\begin{array}{c}
\|y-y_s\|_{W^{1, \infty}(0, T; L^q(\Omega)) \cap L^{\infty}(0, T; W^{1, \infty}(\Omega))} \leq C_{y_s, a}, \\
\|e^{-s\alpha}\phi_0^{-i}(y-Y)\|_{C([0, T/2]; L^{q_i}(\Omega)) \cap L^{q_i}(0, T/2; L^{q_{i+1}}(\Omega))} + \\
\|e^{-s\alpha}\phi_0^{-i}(y-y_s)\|_{C([T/2, T]; L^{q_i}(\Omega)) \cap L^{q_i}(T/2, T; L^{q_{i+1}}(\Omega))} \leq C_{y_s, a} \|y_0 - y_s\|_{q_i}.
\end{array}
\label{est2}
\end{equation}
\label{w2i}
\end{proposition}


This is essentially the result we intend to prove, but the conditions imposed on initial data and on the nonlinear function $a$ are stronger than necessary. Several improvements are possible, as we shall see henceforth.

In the course of proving Proposition \ref{w2i} we use the following local regularity result:
\begin{proposition}
For any sufficiently small $\zeta$ there exists $\rho>0$ such that if
\be
\|z-y_s\|_{W^{1, \infty}(0, T; L^q(\Omega)) \cap L^{\infty}(0, T; W^{1, \infty}(\Omega))} \leq \zeta
\ee
and $\|\Delta a(y_0) + f\|_q \leq \rho$, $\|u\|_{W^{1, \infty}(0, T; L^q(\Omega))} \leq \rho$, $\|z-y_s\|_{L^{\infty}(0, T; L^2(\Omega))} \leq \rho$, and if $T \leq T_0$ with $T_0$ independent of $\zeta$, then the solution $y$ of the linear equation (\ref{y}) satisfies 
\be
\|y-y_s\|_{W^{1, \infty}([0, T]; L^q(\Omega)) \cap L^{\infty}(0, T; W^{1, \infty}(\Omega))} \leq \zeta.
\ee
\label{propnoua}
\end{proposition}

Taking $u=0$, one retrieves a similar statement concerning the solution $Y$ of (\ref{omogen}), which is also used in the proof:
\begin{corollary}
For any sufficiently small $\zeta$ there exists $\rho>0$ such that if $\|\Delta a(y_0) + f\|_q \leq \rho$, then the solution $Y$ of the uncontrolled equation (\ref{omogen}) satisfies
\be
\|Y-y_s\|_{W^{1, \infty}(0, T; L^q(\Omega)) \cap L^{\infty}(0, T; W^{1, \infty}(\Omega))} \leq \zeta.
\ee
\label{cornou}
\end{corollary}

The proof of Proposition \ref{w2i} goes as follows: Assuming that $z \in B$, we first show that the functional $Q_z(y, u)$ has a minimizing pair $(y, u)$ in $L^2_{loc}(Q)$, where $u = m p e^{2s\alpha} s^3 \lambda^3 \phi^3$ and $y$, $p$ satisfy the equation
\begin{equation}
\begin{array}{l}
p_t + \div(b \dl p) = e^{-2s\alpha} ((y-Y) \chi_{[0, T/2]} + (y - y_s) \chi_{[T/2, T]}) \mbox{ on } Q\\
p \mid_\Sigma = 0,\ p(T) = 0.
\end{array}
\label{dual}
\end{equation}
Then, we prove that both $y=F(z)$ and $u$ have further regularity properties, as needed. Define
\begin{equation}
\begin{split}
B_i(\zeta_i) = \{y \in L^2(Q) \mid & \|e^{-s\alpha} \phi_0^{-i} \chi_{[0, T/2]} (y-Y)\|_{L^{q_i}(0, T; L^{q_{i+1}}(\Omega))} + \\
& + \|e^{-s\alpha} \phi_0^{-i} \chi_{[T/2, T]} (y-y_s)\|_{L^{q_i}(0, T; L^{q_{i+1}}(\Omega)} \leq \zeta_i\}.
\end{split}
\end{equation}
We show that $y \in B_0(\zeta_0)$ for some sufficiently small $\zeta_0$ and for a specific choice of $s$ and $\lambda$. Then, we recursively choose $\zeta_i$ as a function of $\zeta_{i-1}$ and prove by induction that $y \in B_i(\zeta_i)$ and $\|e^{-s\alpha} \phi_0^{-i+1} p\|_{L^q(0, T; L^{q+1}(\Omega))} \leq C \zeta_i$.

Following a sufficiently large number of steps, depending on the dimension $n$, we obtain by Sobolev embedding that $\|u\|_{L^{\infty}(Q) \cap W^{1, \infty}(0, T; L^q(\Omega))} \leq C (\|y_0-y_s\|_q + \zeta_0)$. The final step is to choose $\zeta$ and prove that $\|y\|_{W^{1, \infty}(0, T; L^q(\Omega)) \cap L^{\infty}(0, T; W^{1, \infty}(\Omega))} \leq \zeta$. Therefore, $y \in B$.

Thus $B$ is stable under the map $F$. After showing that $F$ has a closed graph, one can apply Schauder's theorem and conclude the proof of Proposition \ref{w2i}.

The goal of proving controllability is achieved in Proposition \ref{w2i}, but only under the condition that $a'$, $a''$, and $1/a'$ should be bounded on $\set R$ and for a less than optimal class of initial data. Indeed, the condition that $\|\Delta a(y_0) + f\|_q$ should be small implies that $a(y_0) \in W^{2, q}(\Omega)$ and therefore $y_0 \in W^{2, q}(\Omega)$.

We relax the conditions on the initial data by means of the following two smoothing estimates, concerning the uncontrolled solution $Y$ of (\ref{omogen}).
\begin{proposition}
For each $\eta>0$ there exists $\rho>0$ such that, if $\|y_0-y_s\|_{W^{1, n}_0(\Omega)} \leq \rho$ and $T \leq \rho$, then the solution $Y$ of equation (\ref{omogen}) satisfies $\|Y-y_s\|_{C([0, T]; W^{1, n}_0(\Omega))} \leq \eta$.
\label{prop1}
\end{proposition}

\begin{proposition}
Consider $q>n$. If $\|y_0-y_s\|_{W^{1, n}_0(\Omega)}$ is sufficiently small, then the solution $Y$ of equation (\ref{omogen}) satisfies
\begin{equation}
\|tY_t\|_{L^{\infty}(0, T; L^q(\Omega))} \leq C_T \|y_0 - y_s\|_2^{\frac 2 q}
\end{equation}
and in particular
\begin{equation}
\|\Delta a(Y(T)) + f\|_q \leq C_T \|y_0 - y_s\|_2^{\frac 2 q}
\end{equation}
for any sufficiently small $T>0$.
\label{prop2}
\end{proposition}

We achieve the desired improvement of Proposition \ref{w2i} in a very basic manner. Namely, let the control $u$ be $0$ on some initial time interval $[0, T_0]$, with $T_0<T$. On the interval $[0, T_0]$, the solution of equation (\ref{dif}) will coincide with the uncontrolled solution $Y$, turning $W^{1, n}_0(\Omega)$ initial data at time $0$ into one of suitable regularity at time $T_0$. Then, one can apply Proposition \ref{w2i} on the interval $[T_0, T]$ and achieve controllability. In this manner, the condition imposed on the initial data can be relaxed to $\|y_0-y_s\|_{W^{1, n}_0(\Omega)}$ being sufficiently small.

Finally, assume that $\|y_0-y_s\|_{C(\ov \Omega)}$ is also small. Then $\|y-y_s\|_{C(\ov Q)}$ can be decreased at will, for $y$ being the controlled solution of equation (\ref{dif}). Therefore the bounds on the nonlinearity $a$ in (\ref{dif}) only matter on a neighborhood of $y_s(\Omega)$. The global boundedness conditions on $a'$ and $a''$ can be eliminated. This leads to the main result, Theorem \ref{2nd}.

%

In the sequel we give detailed proofs of the statements made in this outline, beginning with the auxiliary results and ending with the main one.

\section{Proof of Auxiliary Results}
\subsection{Carleman's Inequality}


\begin{proof}[Proof of Proposition \ref{car}] The subsequent proof essentially follows that given in~\cite{fursikov}.

Recall the auxiliary functions $\alpha$ and $\phi$ given by (\ref{def_ap}) and note that they satisfy
\begin{equation}
\alpha_t \leq C \gamma(\lambda) \phi^2, \alpha_{tt} \leq C \gamma(\lambda) \phi^3,
\end{equation}
where $\gamma(\lambda) = e^{2\lambda \|\psi\|_{C(\overline{\Omega})}}$.

By taking $z = e^{s \alpha} p$ in the dual system (\ref{p}) we obtain the equation
\begin{equation}
\begin{aligned}
&z_t - s \alpha_t z + div(b \dl z) - (2s \lambda \phi b \dl \psi) \dl z + &\\
&+ (s^2 \lambda^2 \phi^2 -s \lambda^2 \phi) b (\dl \psi)^2 z - &\\
&- s \lambda \phi b \Delta \psi z - s \lambda \phi (\dl \psi \cdot \dl b) z = e^{s\alpha} g&\\
&z(0) = z(T) = 0\ \text{on}\ \Omega&\\
&z\mid_{\Sigma} = 0.&
\end{aligned}
\label{z}
\end{equation}

Everywhere in the sequel we assume that $s \geq \gamma(\lambda) = e^{2\lambda \|\psi\|_{C(\overline{\Omega})}}$ and $\lambda \geq 1$. Define
\begin{equation}
D(s, \lambda, z) = \int_Q {s^3 \lambda^3 \phi^3 z^2 + s \lambda \phi (\dl z)^2 dx dt}.
\end{equation}
We see that
\begin{equation}
\|\dl (s^{1/2} \lambda^{1/2} \phi^{1/2} z)\|_{L^2(Q)}^2 \leq C D(s, \lambda, z).
\end{equation}
Hence we also obtain
\begin{equation}
\|s \lambda \phi z\|_{L^2(0, T; L^{\frac {2n}{n - 1}}(\Omega))}^2 \leq C D(s, \lambda, z).
\end{equation}

\begin{lemma}
\begin{multline}
\|s^{-1/2} \lambda^{-1/2} \phi^{-1/2} \dl z\|_{L^2(0, T; L^{\frac{2n}{n - 2}}(\Omega))}^2 \leq \\
\leq C  (1 + \|\dl b|_{L^{\infty}(0, T; L^{2n}(\Omega))}^2 + \|b_t\|_{L^{\infty} (0, T; L^n(\Omega))}^2) D(s, \lambda, z) + C \|e^{s \alpha} g\|_{L^2(Q)}^2.
\label{reg}
\end{multline}
\label{regularitate}
\end{lemma}
\begin{proof}
By squaring equation (\ref{z}), multiplying it by $s^{-1} \lambda^{-1} \phi^{-1}$, and integrating on $Q$, we find that
\begin{multline}
\int_Q {s^{-1} \lambda ^{-1} \phi^{-1} (z_t + \div(b \dl z))^2 dx dt} \leq \\
\leq C (D(s, \lambda, z) + \int_Q {s \lambda \phi (\dl b)^2 z^2 dx dt} + \|e^{s \alpha} g\|_{L^2(Q)}^2).
\end{multline}
On the other hand,
\begin{equation}
\begin{array}{c}
\displaystyle{
\int_Q {2s^{-1} \lambda^{-1} \phi^{-1} z_t \div(b \dl z) dx dt} = -\int_Q {2 b \dl (s^{-1} \lambda^{-1} \phi^{-1} z_t) \dl z dx dt} \geq }\\
\displaystyle{
\geq -\int_Q {s^{-1} \lambda^{-1} \phi^{-1} z_t^2 dx dt} - C \int_Q {s \lambda \phi^{-1} (\dl z)^2 dx dt} + \int_Q {(b s^{-1} \lambda^{-1} \phi^{-1})_t (\dl z)^2 dx dt} \geq }\\
\displaystyle{
\geq -\int_Q {s^{-1} \lambda^{-1} \phi^{-1} z_t^2 dx dt} - C \|b\|_{L^{\infty}(Q)}^2 D(s, \lambda, z) - C \|b_t\|_{L^{\infty}(0, T; L^n(\Omega))} \cdot }\\
\displaystyle{
\cdot \|s^{-1/2} \lambda^{-1/2} \phi^{-1/2} \dl z\|_{L^2(0, T; L^{\frac {2n}{n - 2}}(\Omega))} \cdot \|s^{-1/2} \lambda^{-1/2} \phi^{-1/2} \dl z\|_{L^2(Q)}. }
\end{array}
\end{equation}
Thus we obtain
\begin{equation}
\begin{array}{c}
\displaystyle{
\int_Q {s^{-1} \lambda ^{-1} \phi^{-1} (\div(b \dl z))^2 dx dt} \leq C (1 + \|\dl b\|_{L^{\infty}(0, T; L^n(\Omega))}^2 D(s, \lambda, z) + }\\
\displaystyle{
+ \epsilon \|s^{-1/2} \lambda^{-1/2} \phi^{-1/2} z\|_{L^2(0, T; H^2(\Omega)}^2 + \frac 1 {\epsilon} \|b_t\|_{L^{\infty} (0, T; L^n(\Omega))}^2 D(s, \lambda, z) + \|e^{s\alpha} g\|_{L^2(Q)}^2). }
\end{array}
\end{equation}
We get
\begin{equation}
\begin{array}{c}
\displaystyle{
\int_Q {(b \Delta (s^{-1/2} \lambda ^{-1/2} \phi^{-1/2} z))^2 dx dt} \leq C (1 + \|\dl b\|_{L^{\infty}(0, T; L^{2n}(\Omega))}^2 + }\\
\displaystyle{
+ \frac 1 {\epsilon} \|b_t\|_{L^{\infty} (0, T; L^n(\Omega))}^2) D(s, \lambda, z) + \epsilon \|s^{-1/2} \lambda^{-1/2} \phi^{-1/2} \dl z\|_{L^2(0, T; L^{\frac {2n}{n - 2}}(\Omega)}^2 + \|e^{s\alpha} g\|_{L^2(Q)}^2). }
\end{array}
\end{equation}
Since we know that $b \geq \mu$, by adjusting $\epsilon$ the conclusion (\ref{reg}) follows immediately.~\end{proof}

An immediate consequence is that
\begin{equation}
\|\dl z\|_{L^2(0, T; L^{\frac {2n}{n - 1}}(\Omega))}^2 \leq C ((1 + \zeta) D(s, \lambda, z) + \|e^{s\alpha} g\|_{L^2(Q)}^2).
\end{equation}

We continue with the proof of Proposition \ref{car}. Setting
\be
\begin{aligned}
B(t)z &= -\div(b \dl z) - (s^2 \lambda^2 \phi^2 b(\dl \psi)^2 + s \lambda^2 \phi b(\dl \psi)^2) z + s \alpha_t z\\
X(t)z &= -2s \lambda^2 \phi b(\dl \psi)^2 z -2s \lambda \phi b \dl \psi \cdot \dl z\\
Z(t)z &= s \lambda \phi (\dl \psi \cdot \dl b) z  + s \lambda \phi b \Delta \psi z,
\end{aligned}
\ee
we rewrite (\ref{z}) as
\be
z_t - B(t)z + X(t)z = Z(t)z + e^{s\alpha} g.
\ee

Then, starting with the relation
\begin{equation}
\begin{aligned}
\frac{d}{dt} \int_\Omega {B(t)z \cdot z dx} = &\int_\Omega {B(t)z_t \cdot z + B(t)z \cdot z_t + B_t(t)z \cdot z dx}\\
= &2\int_\Omega {B(t)z  (B(t)z - X(t)z + Z(t)z + e^{s\alpha} g) dx} + \int_\Omega {B_t(t)z \cdot z dx}
\end{aligned}
\end{equation}
and integrating it on $(0, T)$, we obtain
\begin{equation}
2\int_Q {(B(t)z)^2 dx dt} + 2Y = -2\int_Q {B(t)z (Z(t)z + e^{s\alpha} g) dx dt} - \int_Q {B_t(t)z \cdot z dx dt},
\label{ecy}
\end{equation}
where
\begin{equation}
\begin{aligned}
Y = -\int_Q &
{(2s \lambda \phi b \dl \psi \cdot \dl z + 2s \lambda^2 \phi b (\dl \psi)^2 z)}\\
&
{(\div(b \dl z) + (s^2 \lambda^2 \phi^2 b(\dl \psi)^2 + s \lambda^2 \phi b(\dl \psi)^2 z - s \alpha_t z) dx dt}.
\end{aligned}
\label{defy}
\end{equation}

Then,
\begin{equation}
\begin{aligned}
\int_Q {B_t(t)z \cdot z dx dt} = &\int_Q {b_t (\dl z)^2 dx dt} - \\
&- \int_Q {\big((s^2 \lambda^2 \phi^2 + s \lambda^2 \phi) b(\dl \psi)^2\big)_t z^2 + s \alpha_{tt} z^2 dx dt}.
\end{aligned}
\end{equation}
Since $|\alpha_{tt}| \leq \gamma(\lambda)|\phi_{tt}| \leq C \gamma(\lambda) \phi^3$, (here $C$ denotes a constant independent of $s$, $\lambda$, $z$, $\zeta$, and $g$), we eventually obtain the evaluation
\begin{equation}
\begin{aligned}
\bigg|\int_Q {B_t(t)z \cdot z dx dt}\bigg| \leq & C (1+ \|b_t\|_{L^{\infty}(0, T; L^n(\Omega))} (1 + \zeta)) D(s, \lambda, z) + \\
&+\|b_t\|_{L^{\infty}(0, T; L^n(\Omega))} \|e^{s\alpha} g\|_{L^2(Q)}^2.
\label{B_t}
\end{aligned}
\end{equation}
Furthermore,
\begin{multline}
|2\int_Q {B(t)z (Z(t)z +e^{s\alpha} g)dx dt}| \leq \int_Q {(B(t)z)^2 dx dt} + 2\int_Q {(Z(t)z)^2 + e^{2s\alpha} g^2 dx dt}\\
\leq \int_Q {(B(t)z)^2 dx dt} + C ((1 + \|\dl b\|_{L^{\infty}(Q)}^2) D(s, \lambda, z) + \|e^{s\alpha} g\|_{L^2(Q)}^2.
\label{BZ}
\end{multline}
From (\ref{ecy}), (\ref{B_t}), and (\ref{BZ}) we obtain
\begin{equation}
\begin{array}{c}
\displaystyle{
Y \leq C ((1 + \zeta^2) D(s, \lambda, z) + (1+\zeta)\|e^{s\alpha} g\|_{L^2(Q)}^2). }
\end{array}
\label{yless}
\end{equation}
On the other hand, we have the following lower estimates:
\begin{equation}
\begin{array}{c}
\displaystyle{
\int_Q {(s \lambda \phi b \dl \psi \cdot \dl z + s \lambda^2 \phi b(\dl \psi)^2 z) s \alpha_t z dx dt} =  }\\
\displaystyle{
= -\frac 1 2 \int_Q {\div (s^2 \lambda \phi \alpha_t b \dl \psi) z^2 dx dt} + \int_Q {s^2 \lambda^2 \phi \alpha_t b(\dl \psi)^2 z^2 dx dt}, }
\end{array}
\end{equation}
and therefore
\begin{equation}
\begin{array}{c}
\displaystyle{
\bigg|\int_Q {2(s \lambda \phi b \dl \psi \cdot \dl z + s \lambda^2 \phi b \dl \psi \cdot \dl \psi) s \alpha_t z dx dt}\bigg| \leq }\\
\displaystyle{
\leq C (1 + \|\dl b\|_{L^{\infty}(Q)}) D(s, \lambda, z), }
\end{array}
\label{y1}
\end{equation}
because $|\alpha_t| \leq C \gamma(\lambda) \phi^2$ and $|\alpha_{xt}| \leq C \lambda \phi^2$.

Moreover,
\begin{equation}
\begin{array}{c}
\displaystyle{
-\int_Q {2s \lambda^2 \phi b(\dl \psi)^2 z \cdot \div(b \dl z) dx dt} = \int_Q {2b \dl (s \lambda^2 \phi b(\dl \psi)^2 z) \cdot \dl z dx dt} \geq }\\
\displaystyle{
\geq \int_Q {2s \lambda^2 \phi b(\dl \psi)^2 b(\dl z)^2 dx dt} - C (1 + \|\dl b\|_{L^{\infty}(Q)}) D(s, \lambda, z). }
\end{array}
\label{y4}
\end{equation}
Then,
\begin{multline}
-\int_Q {(2s \lambda \phi b \dl \psi \cdot \dl z) (s^2 \lambda^2 \phi^2 + s \lambda^2 \phi) b(\dl \psi)^2 z dx dt} \geq \\
\geq \int_Q {(3s^3 \lambda^4 \phi^3 + 2s^2 \lambda^4 \phi^2) b^2(\dl \psi)^4 z^2 dx dt} - C (1 + \|\dl b\|_{L^{\infty}(Q)}) D(s, \lambda, z).
\label{y5}
\end{multline}
Finally,
\begin{equation}
\begin{array}{c}
\displaystyle{
-\int_Q {2s \lambda \phi (b \dl \psi \cdot \dl z) \div(b \dl z) dx dt} \geq -\int_{\Sigma} {2s \lambda \phi (b \dl \psi \cdot \dl z) (b \nu \cdot \dl z) d\sigma dt} + }\\
\displaystyle{
+ \int_Q {2s \lambda \phi b \dl z \cdot \dl (b \dl \psi \cdot \dl z) dx dt} + \int_Q {2 s \lambda^2 \phi b^2 (\dl \psi \cdot \dl z)^2 dx dt} - C D(s, \lambda, z). }
\end{array}
\label{y6.1}
\end{equation}
Here
\begin{multline}
\int_Q {2s \lambda \phi b \dl z \cdot \dl (b \dl \psi \cdot \dl z) dx dt} = \int_Q {2s \lambda \phi \sum_{i, j}  b D_i z D_i (b D_j \psi D_j z) dx dt} \\
\begin{aligned}
\geq & \int_Q {s \lambda \phi \dl \psi \cdot \dl (b (\dl z)^2) dx dt} + \\
& + \int_Q {2s \lambda \phi b \sum_{i, j} D_j \psi D_i z (D_i b D_j z - D_j b D_i z) dx dt} - C D(s, \lambda, z).
\end{aligned}
\label{y6.2}
\end{multline}
The first integral can be evaluated in the following manner:
\begin{equation}
\begin{array}{c}
\displaystyle{
\int_Q {s \lambda \phi b \dl \psi \cdot \dl (b(\dl z)^2) dx dt} \geq \int_{\Sigma} {s \lambda \phi (b \nu \cdot \dl \psi) (b(\dl z)^2) d\sigma dt} -}\\
\displaystyle{
- \int_Q {s \lambda^2 \phi b (\dl \psi)^2 b (\dl z)^2 dx dt} - C (1 + \|\dl b\|_{L^{\infty}(Q)}) D(s, \lambda, z).  }
\end{array}
\label{y6.3}
\end{equation}
The second integral can be treated as follows:
\begin{equation}
\left|\int_Q {2s \lambda \phi b \sum_{i, j} D_j \psi D_i z (D_i b D_j z - D_j b D_i z) dx dt}\right| \leq C \|\dl b\|_{L^{\infty}(Q)} D(s, \lambda, \phi).
\label{y6.4}
\end{equation}

Putting together inequalities (\ref{y6.1}), (\ref{y6.2}), (\ref{y6.3}), and (\ref{y6.4}), we arrive at the following estimate:
\begin{multline}
-\int_Q {2s \lambda \phi (b \dl \psi \cdot \dl z) \div(b \dl z) dx dt} \geq \\
\begin{aligned}
\geq &\int_{\Sigma} {-2s \lambda \phi (b \dl \psi \cdot \dl z) (b \nu \cdot \dl z) + s \lambda \phi (b \nu \cdot \dl \psi) (b(\dl z)^2) d\sigma dt} - \\
& - C (1 + \|\dl b\|_{L^{\infty}(Q)}) D(s, \lambda, z).
\end{aligned}
\label{y6}
\end{multline}

Since $\psi > 0$ on $\Omega$ and $\psi = 0$ on $\partial \Omega$, we have $\nu = - \frac{\dl \psi}{|\dl \psi|}$. Moreover, because $z$ satisfies a Dirichlet-type boundary condition, it follows that $\dl z$ is parallel to $\nu$ almost everywhere on $\Sigma$. Therefore,
\begin{equation}
\begin{array}{c}
\displaystyle{
\int_{\Sigma} {s \lambda \phi (-2(b \dl \psi \cdot \dl z) (b \nu \cdot \dl z) + (b \nu \cdot \dl \psi) (b (\dl z)^2)) d\sigma dt} = }\\
\displaystyle{
= \int_{\Sigma} {s \lambda \phi |\dl \psi| b^2(\nu \cdot \dl z)^2 d\sigma dt} \geq 0. }
\end{array}
\label{y6.5}
\end{equation}

Combining the inequalities (\ref{y1}), (\ref{y4}), (\ref{y5}), and (\ref{y6}), and taking into account (\ref{y6.5}), we obtain
\be
\begin{aligned}
Y \geq &\int_Q {s^3 \lambda^4 \phi^3 b^2(\dl \psi)^4 \cdot z^2 + s \lambda^2 \phi b(\dl \psi)^2 b(\dl z)^2 dx dt} - \\
& - C (1 + \|\dl b\|_{L^{\infty}(Q)}) D(s, \lambda, z).
\end{aligned}
\label{ygeq}
\end{equation}
From (\ref{ygeq}) and (\ref{yless}) we get
\begin{multline}
\int_Q {s^3 \lambda^4 \phi^3 b^2(\dl \psi)^4 \cdot z^2 + s \lambda^2 \phi b(\dl \psi)^2 b(\dl z)^2 dx dt} \leq \\
\leq C ((1 + \zeta^2) D(s, \lambda, z) + (1+\zeta) \|e^{s\alpha} g\|_{L^2(Q)}^2).
\end{multline}
Because $b \geq \mu$ and $|\dl \psi| \geq c$ on $\Omega \setminus \omega_0$, by making $\lambda$ sufficiently large ($\lambda \geq \lambda_0(\zeta) = C (1 + \zeta^2)$, as stated in the hypothesis) we obtain
\begin{multline}
\int_Q {s^3 \lambda^3 \phi^3 z^2 + s \lambda \phi (\dl z)^2 dx dt} \leq \\
\leq C ((1 + \frac 1 \lambda) \int_{Q_{\omega_0}} {s^3 \lambda^3 \phi^3 z^2 + s \lambda \phi (\dl z)^2 dx dt} + \|e^{s\alpha} g\|_{L^2(Q)}^2).
\label{ineq_p0}
\end{multline}

Substituting $p$ back into (\ref{ineq_p0}), 
we eventually get
\begin{equation}
\begin{array}{c}
\displaystyle{
\int_Q {e^{2s \alpha}(s^3 \lambda^3 \phi^3 p^2 + s \lambda \phi (\dl p)^2) dx dt} \leq }\\
\displaystyle{
\leq C (\int_{Q_{\omega_0}}{e^{2s \alpha}(s^3 \lambda^3 \phi^3 p^2 + s \lambda \phi (\dl p)^2) dx dt} + \|e^{s\alpha} g\|_{L^2(Q)}^2). }
\end{array}
\label{p3}
\end{equation}

Choose $\chi \in C_c^{\infty}(\omega)$, $\chi \geq 0$, such that $\chi \equiv 1$ on $\omega_0$. If we multiply equation (\ref{p}) by $e^{2s \alpha} \chi \phi p$ and integrate on $Q$ we obtain that
\begin{equation}
\begin{array}{c}
\displaystyle{
\int_Q {e^{2s \alpha} \chi \phi b(\dl p)^2 dx dt} \leq }\\
\displaystyle{
\leq \int_Q {\div (b \dl (e^{2s \alpha} \chi \phi)) p^2 dx dt} + C (\int_{Q_{\omega}} {e^{2s \alpha} s \gamma(\lambda) \phi^3 p^2 dx dt} + \|e^{s\alpha} g\|_{L^2(Q)}^2). }
\end{array}
\end{equation}
Thus we obtain, by making a few more estimates, that
\begin{equation}
\int_{Q_{\omega_0}} {e^{2s \alpha} \phi (\dl p)^2 dx dt} \leq C ((1 + \zeta) \int_{Q_{\omega}} {e^{2s \alpha} s^2 \lambda^2 \phi^3 p^2} + \|e^{s\alpha} g\|_{L^2(Q)}^2).
\label{p4}
\end{equation}

From (\ref{p3}) and (\ref{p4}) we obtain the final result (\ref{Carleman}).
\end{proof}

Proceeding as in \cite{barbu}, one obtains the subsequent corollary to Proposition \ref{car} (the Carleman inequality):
\begin{corollary}
Under the assumptions of Proposition \ref{car}, the following inequalities hold:
\begin{equation}
\int_{\Omega} {p^2(0) dx} \leq C_{s, \lambda} (1 + \zeta) \int_{Q_{\omega}}{e^{2s\alpha} \phi^3 p^2 dx dt} + \int_Q {e^{2s\alpha} g^2 dx dt}
\label{cor_p1}
\end{equation}
and
\begin{equation}
\int_0^{T/2} \int_{\Omega} {p^2 dx dt} \leq C_{s, \lambda} (1 + \zeta) \int_{Q_{\omega}}{e^{2s\alpha} \phi^3 p^2 dx dt} + \int_Q {e^{2s\alpha} g^2 dx dt},
\label{cor_p2}
\end{equation}
where $C_{s, \lambda}$ is a constant that does not depend on $p$, $g$, or $\zeta$, but may depend on $s$, $\lambda$, $T$, $\Omega$, and $\omega$.
\label{cor_p0}
\end{corollary}
\begin{proof} Both inequalities stem from the fact that for $t_1 \leq t_2$
\begin{equation}
\|p(t_1)\|_2^2 + \mu\int_{t_1}^{t_2} {\|\dl p(s)\|_2^2 ds} \leq \|p(t_2)\|_2^2 + \int_{t_1}^{t_2} \int_{\Omega} {|p g| dx dt}.
\label{parara} 
\end{equation}
Integrating (\ref{parara}) in $t_2$, one obtains that
\begin{equation}
\int_{\Omega} {p^2(0) dx dt} + \mu \int_0^{T/2} \int_{\Omega} {(\dl p)^2 dx dt} \leq C(\int_{T/2}^{3T/4} \int_{\Omega} {p^2 dx dt} + \int_Q {g^2 dx dt})
\end{equation}
and taking into account the fact that
\begin{equation}
\int_{T/2}^{3T/4} \int_{\Omega} {p^2 dx dt} \leq C_{s, \lambda} \int_Q {e^{2s\alpha} \phi^3 p^2 dx dt}
\end{equation}
the Carleman inequality (\ref{Carleman}) allows one to infer (\ref{cor_p1}) and (\ref{cor_p2}).
\end{proof}

\subsection{Regularity Results}
In the following we prove estimates concerning the uncontrolled solution $Y$ of (\ref{omogen}), as well as one regarding the solution $y$ of the linear equation (\ref{y}).


It is a well-known fact (see for example \cite{pazy}) that if $mu + f \in L^2(Q)$ and $y_0 \in L^2(\Omega)$, then equation (\ref{dif}) has a unique $C([0, T]; L^2(\Omega)) \cap L^2(0, T; W^{1, 2}(\Omega))$ solution.

It is enough to prove all of the following estimates for the case when $Y \in C^{2+\alpha, 1+\alpha/2}(Q)$ and likewise for $y$. Given a less smooth solution $Y$ of (\ref{omogen}), one can replace $Q$ by a sequence of smooth domains $Q^i$ that exhaust it and $y_0$ and $f$ by sequences of smooth functions $y_0^i$ and $f^i$ that approximate them. The corresponding solutions $Y^i$ are $C^{2+\alpha, 1+\alpha/2}$ smooth by parabolic regularity (Theorem 10.1, Chapter 3 of \cite{ladijenskaia}) and converge to $Y$ in $C(0, T; L^2(\Omega))$. If the $C^{2+\alpha, 1+\alpha/2}(Q)$ solutions $Y^i$ fulfill an estimate in a uniform manner, then so does their limit $Y$.

The smooth initial data $y_0^i$ need to satisfy the compatibility conditions by vanishing on the boundary $\partial \Omega$ in order to generate solutions $Y^i$ of $C^{2+\alpha, 1+\alpha/2}(Q)$ regularity. For this reason one has to assume that $y_0 \in W^{1, n}_0(\Omega)$ in the hypotheses of Propositions \ref{propnoua}, \ref{prop1} and \ref{prop2}, as it is the $W^{1, n}(\Omega)$ limit of a sequence of functions $y_0^i$ that vanish on $\partial \Omega$.

The first estimate, Proposition \ref{Yqq}, is trivial.
\begin{proposition}
\begin{equation}
\|Y-y_s\|_{C(0, 1; L^q(Q))} \leq C \|y_0-y_s\|_q.
\label{Yq}
\end{equation}
\label{Yqq}
\end{proposition}
\begin{proof}
Rewriting the equation as
\begin{equation}
(Y-y_s)_t - \div(a'(Y)\dl(Y-y_s)) = \div((a'(Y)-a'(y_s))\dl y_s),
\end{equation}
multiplying by $|Y-y_s|^{q-2}(Y-y_s)$, and integrating in $t$, one obtains that
\begin{multline}
\sup_{t \in [0, T]} \|Y(t)-y_s\|_q^q + \int_Q {|Y-y_s|^{q-2}(\dl(Y-y_s))^2 dx dt} \leq \\
\leq \|y_0-y_s\|_q^q + C \int_Q {|Y-y_s|^{q-1} |\dl(Y-y_s)|},
\end{multline}
whence estimate (\ref{Yq}) follows by Gronwall's Lemma.
\end{proof}

\begin{proof}[Proof of Proposition \ref{propnoua}]
As stated previously, one can assume without loss of generality that $y \in C^{2+\alpha, 1+\alpha/2}(Q)$ as long as $y_0$ vanishes on the boundary. We also assume that $\rho \leq 1$.
Taking the derivative of equation (\ref{y}) with respect to $t$, we get that
\begin{equation}
y_{tt} - \dl(a'(z)\dl y_t) = mu_t + \dl(a''(z)z_t \dl y),\ y_t(0) = \Delta a(y_0) + f,\ y_t \mid_{\Sigma} = 0.
\label{yt}
\end{equation}
Multiplying by $|y_t|^{q-2}y_t$ and integrating in $t$, we find that
\begin{multline}
\sup_{t \in [0, T]} \|y_t(t)\|_q^q + \int_Q {|y_t|^{q-2} (\dl y_t)^2 dx dt} \leq \\
\leq \|y_t(0)\|_q^q + C_0(\|u_t\|_{L^q(Q)}^q + \int_Q {(a''(z) z_t \dl y)^2 |y_t|^{q-2} dx dt}).
\label{ine}
\end{multline}
Consider the fact, proved below in (\ref{qed}), that
\begin{equation}
\|\dl(y-y_s)\|_{L^{\infty}(Q)} \leq C (\|y_t\|_{L^{\infty}(0, T; L^q(\Omega))} + \rho^{1-\beta_0}).
\end{equation}
Then, the last term of (\ref{ine}) can be bounded as follows:
\begin{multline}
\int_Q {(a''(z) z_t \dl y)^2 |y_t|^{q-2} dx dt} \leq C T \|z_t\|_{L^{\infty}(0, T; L^q(\Omega))}^2 (\|\dl (y-y_s)\|_{L^{\infty}(Q)}^2 + \\
\begin{aligned}
&+ \|\dl y_s\|_{L^{\infty}(Q)}^2) \|y_t\|_{L^{\infty}(0, T; L^q(\Omega))}^{q-2} \\
\leq & C T \zeta^2 \|y_t\|_{L^{\infty}(0, T; L^q(\Omega))}^{q-2} (\|y_t\|_{L^{\infty}(0, T; L^q(\Omega))}^2 + \rho^{2(1-\beta_0)}) + \\
& +C_1 (T \zeta^2)^{q/2} +\frac 1 {2 C_0} \|y_t\|_{L^{\infty}(0, T; L^q(\Omega))}^q,
\label{in}
\end{aligned}
\end{multline}
where $C_0$ was introduced in (\ref{ine}).

A bound on $\|y_t(0)\|_q$ can be obtained from the fact that $u(0) = 0$ (because of the rapid decay manifested in (\ref{u})) and $z(0)=y_0$, so
\begin{equation}
\|y_t(0)\|_q = \|\Delta a(y_0) + f\|_q \leq \rho.
\label{yt0}
\end{equation}
We conclude that, for small $\zeta$,
\begin{equation}
\|y_t\|_{L^{\infty}(0, T; L^q(\Omega))}^q + \int_Q {|y_t|^{q-2} (\dl y_t)^2 dx dt} \leq C (\rho^q + (T \zeta^2 \rho^{2(1-\beta_0)})^{q/2}) + C_0 C_1 (T \zeta^2)^{q/2}.
\label{undoi}
\end{equation}
Let us specify $T$ (left unspecified until now) to be such that $C_0 C_1 T^{q/2} \leq \frac 1 2$, where $C_0$ was introduced in (\ref{ine}) and $C_1$ was introduced in (\ref{in}). Note that $C_0$ and $C_1$ depend only on $y_s$ and $a$.

Then, for sufficiently small $\rho$, we retrieve from (\ref{undoi}) that $\|y_t\|_{L^{\infty}(0, T; L^q(\Omega))} \leq \zeta$.

Now let us prove the desired regularity of $\dl y$. Rewrite equation (\ref{y}) in the following form:
\begin{equation}
\div(a'(z)\dl (y-y_s)) = y_t - mu - \div((a'(z)-a'(y_s))\dl y_s).
\label{special}
\end{equation}
Based on the bounds for the right-hand side, we use the maximum principle to evaluate the left-hand side. We state the needed theorems in an abridged form.

\begin{theorem}[Theorem 8.16, \cite{gilbarg}, p. 191]
Let the operator
\be
Lu=\sum_{i, j} D_i(a_{ij} D_j u) 
\ee
be uniformly elliptical: $\lambda |\xi|^2 \leq \sum_{i, j} a_{ij} \xi_i \xi_j \leq \Lambda ^2 |\xi|^2$. Then the solution to the Dirichlet problem
\be
Lu=g + \sum_i D_i f_i,\ u=0 \text{ on } \partial\Omega
\lb{ec_div}
\ee
fulfills the estimate
\be
\|u\|_{\infty} \leq C(n, q, |\Omega|) \lambda^{-1} (\|f\|_q + \|g\|_{q/2}).
\ee
\end{theorem}

Applied to (\ref{special}), the theorem leads to
\begin{multline}
\|y-y_s\|_{\infty} \leq C(\|y_t\|_q + \|mu\|_q + \|(a'(z)-a'(y_s))\dl y_s\|_q) \leq C(\|y_t\|_q + \rho).
\label{infinit}
\end{multline}

Now we employ the following H\"older estimate for the gradient, Theorem 8.33 of \cite[p. 210]{gilbarg} (again the statement is abridged):
\begin{theorem}
Let $u \in C^{1, \alpha}(\Omega)$ be a weak solution of (\ref{ec_div}) in a $C^{1, \alpha}$ domain $\Omega$. Then we have
\be
\|u\|_{C^{1, \alpha}} \leq C(\|u\|_{\infty} + \|g\|_{\infty} + \|f\|_{C^{0, \alpha}}
\ee 
\end{theorem}

Applied to (\ref{special}) it gives, for any $\beta \in (0, 1)$ (we fix one),
\be\begin{aligned}
\|\dl(y-y_s)\|_{C^{0, \beta}(\Omega)} \leq &C(\|y-y_s\|_{\infty} + \|y_t\|_q + \|mu\|_q + \|(a'(z)-a'(y_s))\dl y_s\|_{C^{0, \beta}(\Omega)}) \\
\leq &C(\|y-y_s\|_{\infty} + \|y_t\|_q + \|mu\|_q + \|z-y_s\|_{C^{0, \beta}(\Omega)} \|\dl y_s\|_{\infty} + \\
& + \|z-y_s\|_{\infty} \|\dl y_s\|_{C^{0, \beta}(\Omega)}) \\
\leq & C(\|y_t\|_q + \rho + \|z-y_s\|_{C^{0, \beta}(\Omega)}),
\label{ciuc}
\end{aligned}\ee
where the constant depends on $\|a'(z)\|_{C^{\beta}(\Omega)} \leq C (1 + \|\dl (z-y_s)\|_{\infty})$, so is bounded whenever $\|\dl (z-y_s)\|_{\infty} \leq \zeta \leq 1$. Note that for $\beta_0 = \frac {n+2\beta}{n+2}$
\begin{equation}
\|z-y_s\|_{C^{0, \beta}(\Omega)} \leq C \|z-y_s\|_{W^{1, \infty}(\Omega)}^{\beta_0} \|z-y_s\|_2^{1-\beta_0} \leq C \zeta^{\beta_0} \rho^{1-\beta_0}.
\lb{interpol}
\end{equation}
We prove this interpolation inequality as follows: firstly, we extend functions defined on $\Omega$ continuously by zero to the whole of $\set R^n$. Then, we prove it for $L^{\infty}$, i.\ e.\ $\beta=0$:
\be\begin{aligned}
Cr^n |f(x_0)| \leq &\int_{|x-x_0| \leq r} f(x) \dd x + \int_{|x-x_0| \leq r} |f(x_0) - f(x)| \dd x \\
\leq &C (r^{n/2} \|f\|_2 + r^{n+1} \|f\|_{W^{1, \infty}(\Omega)}).
\end{aligned}
\ee
Setting $r=(\|f\|_2/\|f\|_{W^{1, 1}(\Omega)})^{2/(n-2)}$ we get
\be
\|f\|_{\infty} \leq C \|f\|_{W^{1, \infty}(\Omega)}^{n/(n-2)} \|f\|_2^{2/(n-2)}.
\ee
Now we can interpolate between $\beta=0$ and $1$ and retrieve (\ref{interpol}):
\be
|f(x_1) - f(x_2)|^{1/\beta} \leq C \|f\|_{\infty}^{(1-\beta)/\beta} \|f\|_{W^{1, \infty}} |x_1-x_2|.
\ee
Reintroducing time, since $\rho \leq 1$, we obtain from (\ref{ciuc}) that
\begin{equation}
\|\dl(y-y_s)\|_{L^{\infty}(Q)} \leq C (\|y_t\|_{L^{\infty}(0, T; L^q(\Omega))} + \rho^{1-\beta_0}).
\label{qed}
\end{equation}
By making $\rho$ sufficiently small we obtain the estimate concerning $\dl(y-y_s)$.
\end{proof}

\begin{proof}[Proof of Corollary \ref{cornou}]
In equation (\ref{y}) set $u=0$. Consider the mapping $G$ that associates to each $z$, in the set
\begin{equation}
K = \{z \in L^2(Q) \mid \|z-y_s\|_{W^{1, \infty}([0, T]; L^q(\Omega)) \cap L^{\infty}(0, T; W^{1, \infty}(\Omega))} \leq \zeta\},
\end{equation}
the corresponding solution $y \in K$ of equation (\ref{y}) that exists by Proposition \ref{propnoua} for sufficiently small $\rho$. $K$ is convex and compact in $L^2(Q)$. $G$ has a closed graph in $C(\ov Q)$ and by Schauder's Theorem has a fixed point. The fixed point is the unique solution $Y$ of equation (\ref{omogen}) and satisfies the desired estimate since $Y \in K$.
\end{proof}


\begin{proof}[Proof of Proposition \ref{prop1}]
As justified above, one can assume that $Y \in C^{2+\alpha, 1+\alpha/2}(Q)$. Then, for the derivatives of $Y$ one has the equation
\begin{equation}
\begin{array}{l}
(D_i (Y - y_s))_t = D_i \div (a'(Y) \dl (Y-y_s)) + D_i \div (a'(Y) \dl y_s) + D_i f\\
D_i (Y - y_s) \mid_{t = 0} = D_i (y_0-y_s).
\end{array}
\label{deriv}
\end{equation}
The main problem is that there are no conditions on the lateral boundary $\Sigma$. In the course of this proof we denote $B=a'(Y)$ and $w=Y-y_s$ for brevity.
\begin{lemma}
The solution $Y$ of equation (\ref{omogen}) satisfies
\begin{equation}
\begin{array}{c}
\displaystyle{
\frac d {dt} (\sum_i \|D_i (Y(t)-y_s)\|_n^n) + c \int_{\Omega} {\sum_i |D_i (Y(t)-y_s)|^{n - 2} (\dl D_i (Y(t)-y_s))^2 dx} \leq }\\
\displaystyle{
\leq C (1 + \int_{\Omega} {\sum_i |D_i (Y(t)-y_s)|^{n + 2} dx}). }
\end{array}
\label{ineg}
\end{equation}
\label{semi1}
\end{lemma}
\begin{proof}
By equation (\ref{deriv}),
\begin{equation}
\begin{aligned}
\frac d {dt} \|D_i w(t)\|_n^n = & n \int_{\Omega} |D_i w(t)|^{n - 2} D_i w(t) (D_i \div(B \dl w(t)) + \\
&+ D_i \div(B \dl y_s) + D_i f) dx \\
= & n \int_{\Omega} {|D_i w(t)|^{n - 2} D_i w(t) (D_i \div(B \dl w(t)) dx}\\
& - n(n-1) \int_{\Omega} {|D_i w(t)|^{n - 2} D^2_i w(t) (\div (B \dl y_s) + f) dx}.
\end{aligned}
\label{babel}
\end{equation}
Next we prove, following the same method as in \cite{ladijenskaia2}, Chapter II, Paragraph 6 (also see p. 100), that
\begin{equation}
\begin{aligned}
&-\int_{\Omega} {\sum_i |D_i w(t)|^{n - 2} D_i w(t) D_i \div(B \dl w(t)) dx} + \\
&+ C \Big(\int_{\Omega} {\sum_i |D_i w(t)|^{n - 2} |\dl D_i w(t)|  |\dl w(t)| |D_i B(t)| dx} + \\
&+ \int_{\partial \Omega} {|w(t)|^{n + 1} + |w(t)|^n + |B(t)|^{n + 1} d\sigma}\Big) \geq \\
& \geq (n - 1) \int_{\Omega} {\sum_i |D_i w(t)|^{n - 2} B(\dl D_i w(t))^2 dx}.
\end{aligned}
\label{c}
\end{equation}

Indeed, by Green's Theorem we obtain that
\begin{multline}
-\int_{\Omega} {\sum_i |D_i w(t)|^{n - 2} D_i w(t) D_i \div(B \dl w(t)) dx} = \\
\begin{aligned}
= & -\int_{\partial \Omega} \sum_i |D_i w(t)|^{n - 2} D_i w(t) \nu \cdot D_i (B \dl w(t)) d\sigma + \\
& + (n - 1) \int_{\Omega} \sum_i |D_i w(t)|^{n - 2} \dl D_i w(t) D_i (B \dl w(t)) dx \\
\geq & -\int_{\partial \Omega} {\sum_i |D_i w(t)|^{n - 2} D_i w(t) B \nu  \cdot \dl D_i w(t) d\sigma} + \\
& + (n- 1) \int_{\Omega} {\sum_i |D_i w(t)|^{n - 2} B(\dl D_i w(t))^2 dx} - \\
& - C (\int_{\Omega} {\sum_i |D_i w(t)|^{n - 2} |\dl D_i w(t)|  |\dl w(t)| |D_i B(t)| dx} + \\
& + \int_{\partial \Omega} {\sum_i |D_i w(t)|^{n + 1} + |D_i B(t)|^{n + 1} d\sigma}).
\end{aligned}
\label{c0}
\end{multline}

By introducing local coordinates on the boundary and using a partition of unity we obtain, following the intricate computations given below (very similar to the ones in \cite{ladijenskaia2}), that
\begin{multline}
-\int_{\partial \Omega} {\sum_i |D_i w(t)|^{n - 2} D_i w(t) B \nu \cdot D_i \dl w(t) d\sigma} \geq \\
\geq -C \int_{\partial \Omega} {\sum_i |D_i w(t)|^{n + 1} + |D_i w(t)|^n + |D_i B(t)|^{n + 1} d\sigma}.
\label{cultt}
\end{multline}
Let us prove (\ref{cultt}). Note that on $\Sigma$, since $y=y_s=0$, equation (\ref{omogen}) translates into $\div (B \dl w) = w_t = 0$. By multiplying $\div (B \dl w)$ with $|D_i w|^{n - 2} D_i w \nu_i$ and adding it to the left side of (\ref{cultt}), we obtain
\begin{multline}
-\int_{\partial \Omega} {\sum_i |D_i w(t)|^{n - 2} D_i w(t) B \nu \cdot D_i \dl w(t) d\sigma} \geq \\
\begin{aligned}
\geq & -\int_{\partial \Omega} {\sum_i |D_i w(t)|^{n - 2} D_i w(t) B (\nu \cdot D_i \dl w(t) - \nu_i \Delta w(t)) d\sigma} - \\
& - C \int_{\partial \Omega} {\sum_i |D_i w(t)|^{n + 1} + |D_i B(t)|^{n + 1} d\sigma}.
\end{aligned}
\label{unu}
\end{multline}
Then, because of the smoothness of $\Omega$, each point $x_0$ on the boundary has a neighborhood $V(x_0) \subset \set R^n$ on which we can define a $C^2$ coordinate map $x = S(\tilde x)$, with $S(0) = x_0$ and $S^{-1} (\partial \Omega \cap V(x_0) \subset \set R^{n - 1}$, where $\set R^{n - 1} = \{\tilde x | \tilde x_n = 0 \}$. One can also assume that $|\dl S(\tilde x)| = 1$ on $S^{-1} (\partial \Omega \cap V(x_0)$. Because $\Omega$ is bounded, we can define a smooth partition of unity $\chi = \sum_{j = 1}^J \chi_j$ such that $\chi \mid_{\Omega} = 1$, $\supp \chi_j \subset V(x_j)$ for $j = \overline{1, J - 1}$, and $\supp \chi_J \subset \subset \Omega$. Let us decompose the second $D_i w$ factor on the second line in (\ref{unu}) into a sum according to this partition. After applying Green's theorem, the term corresponding to $\chi_J w(t)$ cancels on the boundary. Putting $w_j = w(t) \chi_j$, $v_j = w_j \circ S$, $v = w(t) \circ S$, each of the other terms becomes by a change of coordinates
\begin{equation}
\begin{array}{c}
\displaystyle{
-\int_{\partial \Omega} {\sum_i |D_i w(t)|^{n - 2} D_i w(t) B (\nu \cdot D_i \dl w_j - \nu_i \Delta w_j) d\sigma} = }\\
\displaystyle{
= -\int_{\set R^{n - 1}} {\sum |S_{m, i} D_m v|^{n - 2} S_{p, i} D_p v B \cdot} }\\
\displaystyle{
{\cdot (S_{n, k} (S_{s, i} S_{r, k} D_s D_r v_j - S_{r, ik} D_r v_j) - S_{n, i} (S_{u, k} S_{l, k} D_u D_l v_j + (\Delta S_u) D_u v_j)) d\sigma}, }
\end{array}
\end{equation}
where the indices $i$, $k$, $l$, $m$, $p$, $s$, $r$, $u$, $v$, and $w$ are used for summation. Since $v = 0$ on $\set R^{n-1}$, we see that $D_i v = 0$ and $D_i D_j v = 0$ on $\set R^{n - 1}$, for $i, j = \overline{1, n - 1}$. In addition, when $s = u = n$ and $r = l$, the corresponding terms cancel. If we put $\sum_k (S_{n, i} S_{q, i} S_{n, k}^2 - S_{n, i}^2 S_{q, k} S_{n, k}) = A_{qi}$ and $(\sum_k S_{n, i} S_{n, k} S_{n, ik} - S_{n, i}^2 \Delta S_n) = A_i$, the previous integral becomes
\begin{equation}
-\int_{\set R^{n - 1}} {\sum |S_{n, i} D_n v|^{n - 2} D_n v(t) (A_{si} b D_q D_n v_j + A_i B D_n v_j) d\sigma}.
\end{equation}
Then, since $A_i$ are bounded and $v$, $v_j$ vanish at infinity,
\begin{equation}
\begin{array}{c}
\displaystyle{
-\int_{\set R^{n - 1}} {\sum |S_{n, i} D_n v|^{n - 2} D_n v A_i B D_n v_j d\sigma} \geq -C \int_{\partial \Omega} {|\dl w(t)|^q d\sigma}. }
\end{array}
\end{equation}
Since $\chi_j$ has compact support, we integrate by parts the term containing $A_{qi}$ and obtain
\begin{equation}
\begin{array}{c}
\displaystyle{
-\int_{\set R^{n - 1}} {\sum |S_{n, i} D_n v|^{n - 2} D_n v A_{si} B (\chi_j D_q D_n v + D_q \chi_j D_n v) d\sigma} \geq }\\
\displaystyle{
\geq \frac 1 n \int_{\set R^{n - 1}} {\sum |S_{n, i} D_n v|^{n-2} D_q (\chi_j A_{si} B) d\sigma} - C \int_{\set R^{n-1}} {|D_n v|^q d\sigma} }\\
\displaystyle{
\geq -C \int_{\partial \Omega} {|\dl w(t)|^n (1 + |\dl w(t)| + |\dl B|) d\sigma}, }
\end{array}
\label{cult}
\end{equation}
which concludes the proof of (\ref{cultt}). Combining inequalities (\ref{c0}) and (\ref{cultt}), inequality (\ref{c}) follows.

Due to the fact that $B \geq \mu > 0$, (\ref{c}) becomes
\begin{multline}
C \int_{\Omega} {\sum_i |D_i w(t)|^{n - 2} (\dl D_i w(t))^2 dx} \leq \\
\begin{aligned}
\leq & -\int_{\Omega} {\sum_i |D_i w(t)|^{n - 2} D_i w(t) D_i \div(B \dl w(t)) dx} + \\
& + C (\int_{\Omega} {|D_i w(t)|^{n - 2} |\dl D_i w(t)| |\dl w(t)| |D_i B(t)| dx} + \\
& + \int_{\partial \Omega} {\sum_i |D_i w(t)|^{n + 1} + |D_i w(t)|^n + |D_i B(t)|^{n + 1} d\sigma}).
\end{aligned}
\end{multline}
For any $\epsilon > 0$,
\begin{equation}
\begin{array}{c}
\displaystyle{
\int_{\Omega} {|D_i w(t)|^{n - 2} |\dl D_i w(t)| |\dl w(t)| |D_i B(t)| dx} \leq \epsilon \int_{\Omega} {\sum_i |D_i Y(t)|^{n - 2} (\dl D_i Y(t))^2 dx} + }\\
\displaystyle{
+ C \int_{\Omega} {\sum_i |D_i B(t)|^{n + 2} dx} + C(\epsilon) \int_{\Omega} {\sum_i |D_i Y(t)|^{n + 2} dx}. }
\end{array}
\end{equation}
Finally, putting to use the fact that $B=a'(Y)$,
\begin{multline}
\int_{\partial \Omega} {|D_i w(t)|^{n + 1} + |D_i w(t)|^n + |D_i B(t)|^{n + 1} d\sigma} \leq \\
\begin{aligned}
\leq & C (\||D_i w(t)|^{n+1}\|_{W^{1, 1}(\Omega)} + \||D_i w(t)|^q\|_{W^{1, 1}(\Omega)} + 1)\\
\leq &\epsilon \int_{\Omega} {\sum_i |D_i w(t)|^{n - 2} (\dl D_i w(t))^2 dx} + C(\epsilon) \int_{\Omega} {|D_i Y|^n + |D_i Y|^{n + 2} dx} + 1.
\end{aligned}
\end{multline}
The last term in (\ref{babel}) can be dealt with as follows:
\begin{equation}
\begin{array}{c}
\displaystyle{
\int_{\Omega} {|D_i w(t)|^{n - 2} D^2_i w(t) (\div (B \dl y_s) + f) dx} \leq \epsilon \int_{\Omega} {\sum_i |D_i w|^{n - 2} (\dl D_i w)^2 dx} + }\\
\displaystyle{
+ C(\epsilon) (\int_{\Omega} {\sum_i |D_i Y|^n dx} + 1). }
\end{array}
\end{equation}

We obtain the desired conclusion (\ref{ineg}).
\end{proof}

From (\ref{ineg}) one can further infer that
\begin{equation}
\begin{array}{c}
\displaystyle{
\frac d {dt} (\sum_i \|D_i (Y(t)-y_s)\|_n^n) + c \int_{\Omega} {\sum_i |D_i (Y(t)-y_s)|^{n - 2} (\dl D_i (Y(t)-y_s))^2 dx} \leq }\\
\displaystyle{
\leq C (1 + (\int_{\Omega} {\sum_i |D_i (Y(t)-y_s)|^{n - 2} (\dl D_i (Y(t)-y_s))^2 dx}) \sum_i \|D_i (Y(t)-y_s)\|_n^2). }
\end{array}
\label{inegal}
\end{equation}
Consider the interval $I=\{t \in [0, T] \mid \sum_i \|D_i (Y(s) - y_s)\|_n^n \leq \eta^n, \forall s \leq t\}$. Since $Y \in C([0, T]; L^2(\Omega))$ and the set $\{y \in L^2(\Omega) \mid y\mid_{\partial \Omega}=0, \sum_i \|D_i y\|_n^n \leq \eta^n\}$ is closed, $I$ is closed as well. Since $0\in I$ by hypothesis, $I$ is nonempty. Let us prove that if $t \in I$ and $t<T$, then $[t, t+\delta) \subset I$ for some $\delta>0$ and consequently $I=[0, T]$.

Since $Y \in C^{2+\alpha, 1+\alpha/2}(Q)$, the derivatives $D_i Y$ are continuous on $Q$. On a right neighborhood of $t$ $[t, t+\delta)$ one has that $\sum_i \|D_i (Y(s)-y_s)\|_n^n \leq 2 \eta^n$. If $\eta$ is sufficiently small, the last term in (\ref{inegal}) cancels and on the interval $[0, t+\delta)$ the inequality (\ref{inegal}) becomes $\sum_i \|D_i (Y(s)-y_s)\|_n^n) \leq \rho^q + C T$. By choosing $\rho$ and $T$ sufficiently small, one retrieves the desired statement, $\sum_i \|D_i (Y(s)-y_s)\|_n^n \leq \eta^n$, first on the neighborhood $[t, t+\delta)$ and then on the whole interval $[0, T]$.
\end{proof}

With the help of Proposition \ref{prop1} we prove a smoothing estimate concerning the solution of the uncontrolled equation (\ref{omogen}), namely Proposition \ref{prop2}.

\begin{proof}[Proof of Proposition \ref{prop2}]
Again, one can assume that $Y \in C^{2+\alpha, 1+\alpha/2}(Q)$. Differentiating equation (\ref{omogen}) with respect to time, we obtain that
\begin{equation}
Y_{tt} - \dl(b \dl Y_t) = \dl(Y_t a''(Y) \dl Y).
\end{equation}
Multiplying by $t$, one obtains that
\begin{equation}
(tY_t)_t - \dl(b \dl (tY_t)) = \dl(t Y_t a''(Y) \dl Y) + Y_t = \dl(t Y_t a''(Y) \dl Y) + \dl(a'(Y) \dl Y - a'(y_s) \dl y_s).
\end{equation}
Multiplying by $|tY_t|^{q-2} tY_t$ and integrating, one obtains that
\begin{multline}
\frac d {dt} \int_{\Omega} {|tY_t(t)|^q dx} + c \int_{\Omega} {|tY_t|^{q-2} (\dl(tY_t))^2 dx} \leq \\
\leq C (\int_{\Omega} {|tY_t|^{q-1} \dl(tY_t) a''(Y) \dl Y  dx} + \int_{\Omega} {|tY_t|^{q-2} \dl(tY_t) (a'(Y) \dl Y - a'(y_s) \dl y_s) dx}).
\end{multline}
Hence,
\begin{multline}
\frac d {dt} \int_{\Omega} {|tY_t(t)|^q dx} + c \int_{\Omega} {|tY_t|^{q-2} (\dl(tY_t))^2 dx} \leq \\
\leq C (\int_{\Omega} {|tY_t|^q (a''(Y) \dl Y)^2 dx} + \int_{\Omega} {|tY_t|^{q-2} ((\dl (Y - y_s))^2 + ((Y-y_s) \dl y_s)^2) dx}).
\end{multline}
Note the following two inequalities:
\begin{equation}
\int_{\Omega} {|tY_t|^q (a''(Y) \dl Y)^2 dx} \leq C (\|\dl y_s\|_{\infty}^2 \|t Y_t\|_q^q + \|Y - y_s\|_{W^{1, n}(\Omega)}^2 \int_{\Omega} {|tY_t|^{q-2} (\dl(tY_t))^2 dx})
\end{equation}
and
\begin{equation}
\int_{\Omega} {|tY_t|^{q-2} ((\dl (Y - y_s))^2 dx} \leq C \|Y-y_s\|_{W^{1, \frac {nq} {n+q-2}}(\Omega)}^2 \bigg(\int_{\Omega} {|tY_t|^{q-2} (\dl(tY_t))^2 dx}\bigg)^{(q-2)/q}.
\end{equation}
They lead, if $\|Y-y_s\|_{C([0, T]; W^{1, n}(\Omega))}$ is smaller than a constant independent of $T$, to
\begin{equation}
\|t Y_t\|_{C([0, T]; L^q(\Omega))} \leq C \|Y-y_s\|_{L^q(0, T; W^{1, \frac {nq} {n+q-2}}(\Omega))}.
\end{equation}
The assumption that $\|Y-y_s\|_{C([0, T]; W^{1, n}(\Omega))} \leq c$ is met by applying Proposition \ref{prop1}. Note further that
\begin{equation}
\|Y-y_s\|_{L^q(0, T; W^{1, \frac {nq} {n+q-2}}(\Omega))} \leq C \|Y-y_s\|_{L^2(0, T; W^{1, 2}(\Omega))}^{\frac 2 q} \|Y-y_s\|_{L^{\infty}(0, T; W^{1, n}(\Omega))}^{\frac {q-2} q}.
\end{equation}
Therefore, under the conditions of Proposition \ref{prop1}, one eventually obtains that
\begin{equation}
\|\Delta a(Y(T)) + f\|_q  = \|Y_t(T)\|_q \leq C \|Y-y_s\|_{L^2(0, T; W^{1, 2}(\Omega))}^{\frac 2 q} \leq C \|y_0 - y_s\|_2^{\frac 2 q}.
\end{equation}
The last inequality follows from writing equation (\ref{omogen}) in the form
\begin{equation}
(Y-y_s)_t - \div(a'(Y)\dl(Y-y_s)) = \div((a'(Y)-a'(y_s)) \dl y_s),
\end{equation}
multiplying it by $Y-y_s$, integrating in $t$, and applying Gronwall's inequality.
\end{proof}

\section{Proof of the Controlability Results}
\subsection{Proposition \ref{w2i}}

We first prove the local controllability of the quasilinear equation (\ref{dif}) under more restrictive conditions on the initial data $y_0$, namely $a(y_0) \in W^{2, q}(\Omega)$, and on $a$, namely $a'$, $a''$, and $1/a'$ bounded on $\set R$.

Proposition \ref{w2i}, will follow by a fixed point argument applied to the linearized equation (\ref{y}). We prove that the continuous mapping $F$ which assigns to each $z$ the minimizer $y$ of the quadratic form $Q_z$ defined by (\ref{F}) has a fixed point in the compact set $B \subset C(\ov Q)$ defined by (\ref{B}).

\begin{proof}[Proof of Proposition \ref{w2i}] For technical reasons we assume that $T \leq T_0$ is sufficiently small. This is not essential, since if the equation is controllable at time $T$ it is also controllable afterwards.

Firstly, assume that $z \in B$, as defined in (\ref{B}), and that $\zeta_i$ and $\zeta$, the constants that enter the definition of $B$, are all at most $1$.

Consider the following optimal control problem: minimize
\begin{equation}
\begin{aligned}
Q_z(y, u) = & \int_{Q_{\omega}} {e^{-2s \alpha} \phi^{-3} u^2 dx dt} + \\
& + \int_Q {  e^{-2s\alpha} (\chi_{[0, T/2]}(y(x, t) - Y(x, t))^2 + \chi_{[T/2, T]} (y(x, t) - y_s(x))^2) dx dt},
\end{aligned}
\end{equation}
subject to (\ref{y}), with $z$ as the auxiliary function.
\begin{lemma}
There exists a pair $(y_1, u_1)$ for which $Q_z(y_1, u_1) < \infty$.
\label{lemma_54}
\end{lemma}
\begin{proof}
Firstly, minimize the functional
\begin{equation}
Q^{\epsilon}_z(y, u) = \int_{Q_{\omega}} {e^{-2s \alpha} \phi^{-3} u^2 dx dt} + \frac 1 {\epsilon} \int_{\Omega} {(y(x, T) - y_s(x))^2 dx}
\end{equation}
subject to (\ref{y}), with $z$ as the auxiliary function.

Clearly $Q^{\epsilon}_z$ is finite for some $(y, u)$ (one only needs to choose a suitable $u$ and solve (\ref{y}) for $y$) and is also strictly convex and lower semicontinuous, so it achieves its minimum. By an argument that will be repeated below, if $(\ye, \ue)$ is the minimizing pair, Pontryagin's principle yields that $\ue = m \pe e^{2s\alpha} \phi^3$, where $\pe$ is a solution to
\begin{equation}
\begin{array}{l}
(\pe)_t + \div(b \dl \pe) = 0 \mbox{ on } Q_T\\
\pe \mid_\Sigma = 0,\ \pe(T) = -\frac 1 {\epsilon} \ye.
\end{array}
\label{dual1}
\end{equation}
Multiplying (\ref{y}) by $\pe$, (\ref{dual1}) by $\ye - y_s$, and adding them together, one obtains
\begin{equation}
\int_{Q_{\omega}} {e^{2s \alpha} \phi^3 \pe^2 dx dt} + \frac 1 {\epsilon} \int_{\Omega} {(\ye(x, T) - y_s(x))^2 dx} = -\int_{\Omega} {y_0 p(0) dx}
\end{equation}
and, taking into account Corollary \ref{cor_p0},
\begin{equation}
\int_{Q_{\omega}} {e^{2s \alpha} \phi^3 \pe^2 dx dt} + \frac 1 {\epsilon} \int_{\Omega} {(\ye(T) - y_s)^2 dx} \leq C_{s, \lambda} (1 + \zeta) \int_{\Omega} {y_0^2 dx}.
\label{majorare1}
\end{equation}
By making $\epsilon$ go to $0$ one obtains in the limit a pair $(y, u)$ satisfying (\ref{y}) and such that $y(T) = 0$.

Performing this procedure on a subinterval $[\epsilon_1, T_1] \subset (0, T)$ instead of $[0, T]$, with the initial data $y(\epsilon_1) = Y(\epsilon_1)$, one obtains a pair $(y_1, u_1)$ with $y_1(T_1) = 0$. One can extend $y_1$ by $Y$ on the interval $[0, \epsilon_1)$ and by $y_s$ on the interval $(T_1, T]$ and extend $u_1$ by $0$ on both intervals, making $y_1$ a solution of (\ref{y}), for the control $u_1$, on the whole of $[0, T]$. By an immediate computation $Q_z(y_1, u_1) < \infty$.
\end{proof}

Since $Q_z$ is a strictly convex, lower semicontinuous functional, which is finite for the pair $(y_1, u_1)$ introduced in Lemma \ref{lemma_54}, it has a unique minimizing pair $(y, u)$. By the Pontryagin principle we obtain that $u = m p e^{2s \alpha} s^3 \lambda^3 \phi^3$, where $p$ is a solution to
\begin{equation}
\begin{array}{l}
p_t + \div(b \dl p) = e^{-2s\alpha} ((y-Y) \chi_{[0, T/2]} + (y - y_s) \chi_{[T/2, T]}) \mbox{ on } Q\\
p \mid_\Sigma = 0,\ p(T) = 0.
\end{array}
\end{equation}

Substracting (\ref{omogen}) from (\ref{y}) and multiplying it by $p$, multiplying (\ref{dual}) by $y - Y$, adding them together, and integrating on $[0, T/2]$, we obtain
\begin{multline}
\int_0^{T/2} \int_{\omega} {e^{2s \alpha} s^3 \lambda^3 \phi^3 p^2 dx dt} + \int_0^{T/2} \int_{\Omega} {e^{-2s\alpha} (y(x, t) - Y(x, t))^2 dx dt} = \\
= \int_{\Omega} {(y(T/2)-Y(T/2)) p(T/2) dx} + \int_0^{T/2} \int_{\Omega} {(a'(z) - a'(Y)) \dl Y \dl p dx dt}.
\label{part1}
\end{multline}
Multiplying (\ref{y}) by $p$, (\ref{dual}) by $y-y_s$, adding them together, and integrating on $[T/2, T]$, we get
\begin{multline}
\int_{T/2}^T \int_{\omega} {e^{2s \alpha} s^3 \lambda^3 \phi^3 p^2 dx dt} + \int_{T/2}^T \int_{\Omega} {e^{-2s\alpha} (y(x, t) - y_s(x))^2 dx dt} = \\
= -\int_{\Omega} {(y(T/2)-y_s) p(T/2) dx} + \int_{T/2}^T \int_{\Omega} {(a'(z) - a'(y_s)) \dl y_s \dl p dx dt}.
\label{part2}
\end{multline}
Adding (\ref{part1}) and (\ref{part2}) together gives
\begin{multline}
\int_{Q_{\omega}} {e^{2s \alpha} s^3 \lambda^3 \phi^3 p^2 dx dt} + \int_Q {e^{-2s\alpha} (\chi_{[0, T/2]} (y-Y)^2 + \chi_{[T/2, T]} (y-y_s)^2) dx dt} = \\
= \int_{\Omega}{(y_s-Y(T/2)) p(T/2) dx} + \\
+ \int_Q {(\chi_{[0, T/2]} (a'(z) - a'(Y)) \dl Y + \chi_{[T/2, T]} (a'(z) - a'(y_s)) \dl y_s) \dl p dx dt}.
\label{majorare2}
\end{multline}

In equation (\ref{dual}) the conditions for Carleman's inequality, Proposition \ref{car}, are met uniformly, since 
\begin{equation}
\|b\|_{L^{\infty}(0, T; W^{1, \infty}(\Omega)) \cap W^{1, \infty}(0, T; L^n(\Omega))} \leq C\|z\|_{L^{\infty}(0, T; W^{1, \infty}(\Omega)) \cap W^{1, \infty}(0, T; L^n(\Omega))} \leq C \zeta.
\end{equation}
With the help of the Carleman inequality, Corollary \ref{cor_p0}, and Corollary \ref{cornou}, we obtain that
\begin{multline}
Q_z(y, u) \leq \|y_s-Y(T/2)\|_2 \|p(T/2)\|_2 + C(s\lambda)^{-1/2} \|\phi^{-1/2}\|_{\infty} \zeta_0 \cdot \\
\cdot \bigg((1+\zeta)\int_{Q_{\omega}} {e^{2s \alpha} s^3 \lambda^3 \phi^3 p^2 dx dt} + \int_Q {e^{-2s\alpha} ((y-Y)^2 \chi_{[0, T/2]} + (y-y_s)^2 \chi_{[T/2, T]}) dx dt}\bigg)^{1/2}
\end{multline}
and therefore, since $\zeta \leq 1$,
\begin{multline}
\int_Q {e^{2s \alpha} s^3 \lambda^3 \phi^3 p^2 dx dt} + \int_Q {e^{-2s\alpha} (\chi_{[0, T/2]} (y-Y)^2 + \chi_{[T/2, T]} (y-y_s)^2) dx dt} \leq \\
\leq C_{s, \lambda} \|y_0-y_s\|_2^2 + C (s\lambda)^{-1} \|\phi^{-1}\|_{\infty} \zeta_0^2.
\label{majorare}
\end{multline}
If in (\ref{majorare}) $s \geq s_1(\lambda) = \max(s_0(\lambda), 2C\lambda^{-1}\|\phi^{-1}\|_{\infty})$ and if $\|y_0-y_s\|_2$ is sufficiently small, $\|y_0-y_s\|_2 \leq C_{s, \lambda} \zeta_0$, then $F$ takes values in the set
\begin{equation}
B_0(\zeta_0) = \{y \in L^2(Q) \mid \|e^{-s\alpha} \chi_{[0, T/2]} (y-Y)\|_{L^2(Q)} + \|e^{-s\alpha} \chi_{[T/2, T]} (y-y_s)\|_{L^2(Q)} \leq \zeta_0\},
\end{equation}
regardless of the size of $\zeta_0$.

Up to this point, constants' dependence on $s$ and $\lambda$ was not assumed and was made explicit wherever it occurred. Throughout the sequel constants will depend on $s$ and $\lambda$ by default.

In the following we prove that $y = F(z) \in B$.

Rewrite equation (\ref{y}) as
\begin{equation}
(y-y_s)_t - \dl (b \dl(y-y_s)) = mu + \dl ((a'(z)-a'(y_s)) \dl y_s),
\label{yy}
\end{equation}
where $b=a'(z)$.
Also consider the equation
\begin{equation}
(y-Y)_t - \dl (b \dl(y-Y)) = mu + \dl ((a'(z)-a'(Y)) \dl Y).
\label{yyy}
\end{equation}
Set $e^{-s \alpha}\phi_0^{-i}(y-Y) = V_i$, $e^{-s \alpha}\phi_0^{-i}(y-y_s) = W_i$, $e^{s \alpha}\phi_0^{-i+1}p = P_i$. 
Equations (\ref{yy}), (\ref{yyy}), and (\ref{dual}) turn into
\begin{equation}
\begin{aligned}
&(V_i)_t - div(b \dl V_i) + 2s \lambda \phi b \dl \psi \cdot \dl V_i + (s \alpha_t + \phi_0^{-i} (\phi_0^i)_t V_i +\\
&+ (s^2 \lambda^2 \phi^2 + s \lambda^2 \phi) b (\dl \psi)^2 V_i + s \lambda \phi b \Delta \psi V_i + s \lambda \phi b (\dl \psi \cdot \dl b) V_i = \\
&= m s^3 \lambda^3 \phi^3 \phi_0^{-1} P_i + e^{-s\alpha} \phi_0^{-i} \dl ((a'(z) - a'(Y)) \dl Y)\\
&V_i(0) = 0,\ V_i\mid_{\Sigma} = 0,
\end{aligned}
\label{YY}
\end{equation}
\begin{equation}
\begin{aligned}
&(W_i)_t - div(b \dl W_i) + 2s \lambda \phi b \dl \psi \cdot \dl W_i + (s \alpha_t + \phi_0^{-i} (\phi_0^i)_t W_i +\\
&+ (s^2 \lambda^2 \phi^2 + s \lambda^2 \phi) b (\dl \psi)^2 W_i + s \lambda \phi b \Delta \psi W_i + s \lambda \phi b (\dl \psi \cdot \dl b) W_i = \\
&= m s^3 \lambda^3 \phi^3 \phi_0^{-1} P_i + e^{-s\alpha} \phi_0^{-i} \dl ((a'(z) - a'(y_s)) \dl y_s)\\
&W_i(T/2) = V_i(T/2) + e^{s\alpha} \phi_0^i (Y(T/2) - y_s),\ W_i\mid_{\Sigma} = 0,
\end{aligned}
\label{Y}
\end{equation}
and
\begin{equation}
\begin{aligned}
&(P_i)_t + div(b \dl P_i) - 2s \lambda \phi b \dl \psi \cdot \dl P_i - (s \alpha_t - \phi_0^{1-i} (\phi_0^{i-1})_t) P_i +\\
&+ (s^2 \lambda^2 \phi^2 - s \lambda^2 \phi) b (\dl \psi)^2 P_i -\\
&- s \lambda \phi b \Delta \psi P_i - s \lambda \phi b (\dl \psi \cdot \dl b) P_i = -\phi_0 (V_i \chi_{[0, T/2]} + W_i \chi_{[T/2, T]})\\
&P_i(0) = P_i(T) = 0\ \text{on}\ \Omega,\ P_i\mid_{\Sigma} = 0.
\end{aligned}
\label{P}
\end{equation}

Recall that by (\ref{def_q}) one has that $q_0=2$, $\ds q_i = 2 (\frac {n}{n-2})^{i-1}$ for $1 \leq i \leq N-1$, and $q_N = q >n$. Assuming that $P_{i-1}$, $V_{i-1} \chi_{[0, T/2]}$, and $W_{i-1} \chi_{[T/2, T]}$ satisfy the bound
\begin{multline}
c\|P_{i-1}\|_{L^{q_{i-1}}(0, T; L^{q_i}(\Omega))} + \|V_{i-1}\|_{L^{q_{i-1}}(0, T/2; L^{q_i}(\Omega))} + \|W_{i-1}\|_{L^{q_{i-1}}(T/2, T; L^{q_i}(\Omega))} \leq \\
\leq C \|y_0-y_s\|_{q_{i-1}} + \frac 1 2 \zeta_{i-1},
\label{boundd}
\end{multline}
we prove that the same holds with $i$ replacing $i-1$. By induction the inequality will hold for all $i = \overline{1, N}$. This implies that $F$ takes values in $B_i(\zeta_i)$ for each $i$, for a judicious choice of $\zeta_i$.

The induction hypothesis is fulfilled since when $i=1$
\begin{multline}
c \|e^{s\alpha} \phi_0 p\|_{L^2(Q)} + \|e^{-s\alpha} (y-Y)\|_{L^2([0, T/2] \times \Omega)} + \|e^{-s\alpha} (y-y_s)\|_{L^2([T/2, T] \times \Omega)} \leq \\
\leq C \|y_0-y_s\|_2 + \frac 1 2 \zeta_0
\end{multline}
follows from (\ref{majorare}).
Multiplying equation (\ref{P}) by $|P_i|^{q_i-2}P_i$, integrating in $t$, and following an integration by parts to get rid of the term containing $\dl P_i$, one obtains the usual estimate
\begin{multline}
\sup_{t\in[0, T]}\|P_i(t)\|_{q_i}^{q_i} + \int_0^T \int_{\Omega} {|P_i|^{q_i-2} (\dl P_i)^2 dx dt} \leq \\
\leq C \int_Q {|P_i|^{q_i - 1} (|P_i| \phi_0^2 + (|V_i| \chi_{[0, T/2]} + |W_i| \chi_{[T/2, T]}) \phi_0) dx dt}
\end{multline}
and therefore
\begin{equation}
\leq C (\|P_{i-1}\|_{L^2(0, T; L^{q_i}(\Omega))}^{q_i} + \|V_{i-1}\|_{L^1(0, T/2; L^{q_i}(\Omega))}^{q_i}) + \|W_{i-1}\|_{L^1(T/2, T; L^{q_i}(\Omega))}^{q_i}).
\label{calcul}
\end{equation}
Indeed, since $P_i = \phi_0^{-1} P_{i-1}$, all the minor terms containing $P_i$ are majorized by the norm of $P_{i-1}$, and the same goes for $V_i$ and $W_i$.

The same holds for $V_i$ and $W_i$, with a few minor differences. Firstly, for $V_i$ one obtains that
\begin{multline}
\sup_{t\in[0, T/2]}\|V_i(t)\|_{q_i}^{q_i} + \int_0^{T/2} \int_{\Omega} {|V_i|^{q_i-2} (\dl V_i)^2 dx dt} \leq \\
\begin{aligned}
\leq C \int_{[0, T/2]\times\Omega} & |V_i|^{q_i-2} (|V_i|^2 \phi^2 + |V_i| |P_i| \phi^2 + |\dl V_i| |z-Y| e^{-s \alpha} \phi_0^{-i} + \\
& + |V_i| |z-Y| e^{-s\alpha} \phi \phi_0^{-i}) dx dt
\end{aligned}
\end{multline}
and therefore
\begin{equation}
\leq C (\|P_{i-1}\|_{L^2(0, T; L^{q_i}(\Omega))}^{q_i} + \|V_{i-1}\|_{L^2(0, T/2; L^{q_i}(\Omega))}^{q_i}) + \|(z-Y)e^{-s\alpha}\phi_0^{1-i}\|_{L^2(0, T/2; L^{q_i}(\Omega))}^{q_i})
\end{equation}

For equation (\ref{Y}) the initial condition at time $T/2$ yields
\begin{equation}
\|W_i(T/2)\|_{q_i} \leq \|V_i(T/2)\|_{q_i} + C \|y_0-y_s\|_{q_i}.
\end{equation}
Then, multiplying equation (\ref{Y}) by $|W_i|^{q_i-2}W_i$, integrating in $t$, and performing two integrations by parts, one obtains that
\begin{multline}
\sup_{t\in[T/2, T]}\|W_i(t)\|_{q_i}^{q_i} + c \int_{T/2}^T \int_{\Omega} {|W_i|^{q_i-2} (\dl Y_i)^2 dx dt} \leq \\
\begin{aligned}
\leq & \|W_i(T/2)\|_{q_i}^{q_i} + C \int_{[0, T/2]\times\Omega} |W_i|^{q_i-2} (|W_i|^2 \phi^2 + |W_i| |P_i| \phi^2 + \\
& + |\dl W_i| |z-y_s| e^{-s \alpha} \phi_0^{-i} + |W_i| |z-y_s| e^{-s\alpha} \phi \phi_0^{-i}) dx dt
\end{aligned}
\end{multline}
and therefore
\begin{multline}
\leq C (\|W_i(T/2)\|_{q_i}^{q_i} + \|P_{i-1}\|_{L^2(0, T; L^{q_i}(\Omega))}^{q_i} + \|W_{i-1}\|_{L^2(T/2, T; L^{q_i}(\Omega))}^{q_i}) + \\
+ \|(z-y_s)e^{-s\alpha}\phi_0^{1-i}\|_{L^2(T/2, T; L^{q_i}(\Omega))}^{q_i}).
\label{uciuciu}
\end{multline}
The conclusion so far is that
\begin{multline}
\|P_i\|_{L^{q_i}(0, T; L^{q_{i+1}}(\Omega))} + \|V_i\|_{L^{q_i}(0, T/2; L^{q_{i+1}}(\Omega))} + \|W_i\|_{L^{q_i}(T/2, T; L^{q_{i+1}}(\Omega))} \leq \\
\leq C (\|y_0-y_s\|_{q_i} + \zeta_{i-1}).
\end{multline}
Letting $\zeta_i = 2C \zeta_{i-1}$, the induction is complete.

Let $\ds N = \left\lceil\frac {\log q - \log 2}{\log{n} - \log{(n-2)}}\right\rceil+1$. After $N$ induction steps one eventually has that
\begin{equation}
\|P_N\|_{L^{\infty}(0, T; L^q(\Omega))} + \|V_N\|_{L^{\infty}(0, T/2; L^q(\Omega))} + \|W_N\|_{L^{\infty}(T/2, T; L^q(\Omega))} \leq C(\|y_0 - y_s\|_q + \zeta_0).
\end{equation}
This method cannot lead to an $L^{\infty}$ bound, so the last step must be different. Denote $P = e^{(s+\delta)\alpha_0} \phi_0^3 p$, where $\delta>0$. Equation \ref{dual} becomes
\begin{equation}
\begin{array}{l}
P_t - \dl(b\dl P) = (e^{(s+\delta)\alpha_0} \phi_0^3)_t p + e^{(s+\delta)\alpha_0-2s\alpha} \phi_0^3 (y-Y) \chi_{[0, T/2]} + (y-y_s) \chi_{[T/2, T]}\\
P(T) = 0, P \mid_{\Sigma}=0.
\end{array}
\end{equation}
Since $\alpha_0 \leq \alpha$, one has that $\|P_t - \dl(b \dl P)\|_{L^q(Q)} \leq C(\|y_0 - y_s\|_q + \zeta_0)$. By parabolic regularity (Theorem 9.1, Chapter IV, of \cite{ladijenskaia}), one obtains that
\be
\|P\|_{W^{2, 1}_q(Q)} \leq C(\|y_0 - y_s\|_q + \zeta_0).
\ee
The constant depends only on the modulus of continuity of $b=a'(z)$, which is bounded. In fact, the constant only depends on $\zeta$, because
\begin{equation}
|(z-y_s)(x_1, t_1) - (z-y_s)(x_2, t_2)| \leq C \zeta (|x_1 - x_2| + |t_1 - t_2|^{\frac 1 2}).
\end{equation}
Since $a'$ has the Lipschitz property and $y_s$ is fixed and Lipschitz continuous, the modulus of continuity is uniformly bounded under the already made assumption that $\zeta \leq 1$.

Since
\begin{equation}
u = m p e^{2s \alpha} s^3 \lambda^3 \phi^3 = m P e^{2s \alpha - (s + \delta) \alpha_0} s^3 \lambda^3 \phi^3 \phi_0^{-3},
\end{equation}
we get that $|u| \leq C e^{(s (1 - 2\eta(\lambda)) - \delta) \alpha_0} |P|$ and $|u_t| \leq C e^{(s (1 - 2\eta(\lambda)) - \delta) \alpha_0} \allowbreak \phi_0 (|P| + |P_t|)$. Note that $W^{2, 1}_q(Q) \subset L^{\infty}(Q)$ is a compact embedding. Finally, we obtain that
\begin{equation}
\|e^{-(s (1 - 2\eta(\lambda)) - \delta) \alpha_0} u\|_{L^{\infty}(Q)} + \|e^{-(s (1 - 2\eta(\lambda)) - \delta) \alpha_0} \phi_0^{-1} u_t\|_{L^q(Q)} \leq C(\|y_0 - y_s\|_q + \zeta_0).
\label{u}
\end{equation}

This leads to the desired $L^{\infty}(0, T; W^{1, \infty}(\Omega)) \cap W^{1, \infty}(0, T; L^q(\Omega))$ bound concerning $y$. Indeed, Proposition \ref{propnoua} yields that if $\|\Delta a(y_0) + f\|_q$, $\|y_0-y_s\|_q$, and $\zeta_0$ are sufficiently small, if $\|z-y_s\|_{W^{1, \infty}(0, T; L^q(\Omega)) \cap L^{\infty}(0, T; W^{1, \infty}(\Omega)} \leq \zeta$, and if $T \leq T_0$ as assumed initially, then $\|y-y_s\|_{W^{1, \infty}(0, T; L^q(\Omega)) \cap L^{\infty}(0, T; W^{1, \infty}(\Omega)} \leq \zeta$ as well.

Thus, $y$ belongs to the compact set $B \subset C(\ov Q)$ and therefore the mapping $F:~B \to B$ is compact. It remains to prove that $F$ has a closed graph. Given a sequence $z_n \in B$ with $F(z_n)=y_n$, consider the corresponding controls $u_n$ and solutions $p_n$ of equation (\ref{dual}), linked by
\be
u_n = m p_n e^{2s\alpha} s^3 \lambda^3 \phi^3.
\ee
Assume that $z_n$ converge to $z$ and $y_n$ converge to $y$ in $C(\ov Q)$. By (\ref{majorare}),
\be
e^{-s\alpha} (s \lambda \phi)^{-3/2} u_n = m p_n e^{s\alpha} (s \lambda \phi)^{3/2}
\ee
are uniformly bounded in $L^2(Q)$. Therefore, after taking a subsequence, $u_n$ and $p_n$ converge weakly to some limits $u$ and $p$ that also fulfill the relation
\be
u = m p e^{2s\alpha} s^3 \lambda^3 \phi^3.
\ee
Since $y_n$ and $p_n$ satisfy equations (\ref{y}) and (\ref{dual}) with $z_n$ as coefficients, the same is true about their limits $y$, $p$, and $z$. Thus, $y$, $p$, and $z$ satisfy in a weak sense the relations
\begin{equation}
\begin{array}{l}
y_t-\div(b \dl y) = m p e^{2s\alpha} s^3 \lambda^3 \phi^3 + f,\ y(0) = y_0,\ y \mid_{\Sigma} = 0\\
p_t + \div(b \dl p) = e^{-2s\alpha} ((y-Y) \chi_{[0, T/2]} + (y - y_s) \chi_{[T/2, T]}),\ p(T) = 0, p \mid_{\Sigma} = 0.
\end{array}
\label{rel}
\end{equation}

Consider $F(z) = y_e$ with the corresponding control $u_e$. By Pontryagin's principle there exists $p_e$ such that $y_e$, $p_e$, and $z$ also satisfy the relations (\ref{rel}). Substracting the two sets of relations from one another and denoting $dy = y-y_e$, $dp = p-p_e$, one obtains that
\begin{equation}
\begin{array}{l}
dy_t-\div(b \dl dy) = m dp e^{2s\alpha} s^3 \lambda^3 \phi^3,\ dy(0) = 0,\ dy \mid_{\Sigma} = 0\\
dp_t + \div(b \dl dp) = e^{-2s\alpha} dy,\ dp(T) = 0, dp \mid_{\Sigma} = 0.
\end{array}
\end{equation}
Multiplying the first equation by $dp$, the second by $dy$, adding them together, and integrating on $[0, T]$ one has that
\begin{equation}
\int_{Q_{\omega}} {(dp)^2 e^{2s\alpha} s^3 \lambda^3 \phi^3 dx dt} + \int_Q {e^{-2s\alpha} (dy)^2 dx dt} = 0.
\end{equation}
Therefore $dy \equiv 0$ and $y=y_e=F(z)$, thus finishing the proof of the fact that $F$ has a closed graph.

Since $F$ is compact and has a closed graph, it is continuous. The proof has already shown that $F$ takes $B$ to itself. The fixed point obtained by applying Schauder's Theorem is the required controlled solution. Indeed, by the definition (\ref{B}) of $B$ it already follows that $y(T)=0$.

By making $\zeta_0$, $\zeta$, and $\|y_0-y_s\|_q + \|\Delta a(y_0) + f\|_q$ sufficiently small, one obtains that all of $\zeta_i$ and $\zeta$ are at most $1$, thus fulfilling the condition imposed at the beginning of the proof.


The estimates (\ref{est1}) and (\ref{est2}) concerning $y$ and $u$ contained in the statement of Proposition \ref{w2i} follow from those already obtained in the course of the proof, if one imposes the supplementary condition that $\zeta_i \leq \|y_0-y_s\|_{q_i}$, $\zeta \leq \|\Delta a(y_0) + f\|_q$. Note that, since we have already assumed that $\|\Delta a(y) + f\|_q$ is bounded, all the weaker norms of $y_0-y_s$ can actually be controlled by $\|y_0-y_s\|_2$.
\end{proof}

\subsection{Theorem \ref{2nd}}

Proposition \ref{w2i} establishes local controllability only for the case when $a(y_0) \in W^{2, q}(\Omega)$. By combining it with the regularity results contained in Propositions \ref{prop1} and \ref{prop2}, now we prove local controllability for the more general case of $y_0 \in W^{1, n}_0(\Omega)$, still under the assumption that $a'$, $a''$, and $1/a'$ are bounded on $\set R$. Then we weaken this assumption as well.
\begin{proof}[Proof of Theorem \ref{2nd}]
In the course of the proof we assume that $T$ is sufficiently small, but this assumption can be removed since if the equation is controllable at time $T$ it is also controllable for any greater time.

Divide the interval $[0, T]$ into two subintervals, $[0, T_0]$ and $[T_0, T]$, of sufficiently small length. On the first interval, take the control $u$ to be $0$. Propositions \ref{Yqq}, \ref{prop1}, and \ref{prop2} yield that the solution of equation
\begin{equation}
\begin{array}{ll}
Y_t - \Delta a(Y) = f & \mbox{on }Q_{T_0} = \Omega \times (0, T)\\
Y \mid_{\Sigma} = 0, & y_0 = y_0
\end{array}
\label{omogena}
\end{equation}
satisfies the estimate
\begin{equation}
\|Y(T_0) - y_s\|_q + \|\Delta a(Y(T_0)) + f\|_q \leq C(T_0, y_s, a) \|y_0 - y_s\|_2^{2/q}
\label{ecu1}
\end{equation}
provided that $\|y_0 - y_s\|_{W^{1, n}_0(\Omega)}$ is sufficiently small.
Then, by applying Proposition \ref{w2i} on the second interval $[T_0, T]$, if $\|y_0-y_s\|_2$ is sufficiently small, one obtains a controlled solution $y^*$ and a control $u^*$ on the interval $[T_0, T]$. By Proposition \ref{w2i}, $y^*$ and $u^*$ satisfy the required inequalities (\ref{est1}) and (\ref{est2}) on the interval $[T_0, T]$. Due to the rapid decrease of $u^*$ and $y^*$ near $T_0$ and the regularity of these functions, $y$ and $u$, where
\begin{equation}
u = \left\{ \begin{array}{ll}
0 & \text{on } [0, T_0] \times \Omega\\
u^* & \text{on } (T_0, T) \times \Omega
\end{array} \right.
\end{equation}
and
\begin{equation}
y = \left\{ \begin{array}{ll}
Y & \text{on } [0, T_0] \times \Omega\\
y^* & \text{on } (T_0, T) \times \Omega,
\end{array} \right.
\end{equation}
are a controlled solution and a control for (\ref{dif}) on $[0, T]$ respectively.

Taking into account Propositions \ref{Yqq} and \ref{prop2}, the estimates (\ref{est1}) and (\ref{est2}) given by Proposition \ref{w2i} for $y^*$ and $u^*$ on the interval $[T_0, T]$ translate into the desired estimates (\ref{control}) and (\ref{control2}) on $[0, T]$.

The only new estimate is the one concerning $\|t(y-y_s)\|_{L^{\infty}(0, T; W^{1, \infty}(\Omega))}$, which can be proved in a manner similar to (\ref{qed}). Here, one obtains that for every $t$
\begin{equation}
\|y-y_s\|_{C^{\beta}(\Omega)} \leq C(\|y_t\|_q + \|y-y_s\|_2^{\beta_1})
\end{equation}
for a fixed $\beta_1$, whence the estimate follows.

Finally, we remove the global boundedness conditions on $a$.

Since $a(y_s) \in W^{2, q}(\Omega)$, $y_s$ is continuous and bounded. Its range is an interval, with endpoints $y_1$ and $y_2$. The set $\{x | d(x, y_s(\Omega)) \leq \epsilon\}$ is the interval $[y_1 - \epsilon, y_2 + \epsilon]$.
\begin{lemma}
If $a \in C^2(\set R)$ and $a'(y)>0$ $\forall y \in \set R$, then there exists a function $A: \set R \rightarrow \set R$ with two bounded derivatives on $\set R$ that agrees with $a$ on $[y_1 - \epsilon, y_2 + \epsilon]$ and has the property that $A'(y) \geq \mu_A > 0$ on $\set R$.
\label{3rd}
\end{lemma}
\begin{proof} Define the function $A_3: \set R \rightarrow \set R$ by
\begin{equation}
A_3(x) = \left\{\begin{array}{ll}
a''(x), & x \in [y_1 - \epsilon, y_2 + \epsilon]\\
a''(y_1 - \epsilon), & x < y_1 - \epsilon\\
a''(y_2 + \epsilon), & x > y_2 + \epsilon.
\end{array}\right.
\end{equation}
Consider the function $A_2 = A_3 \chi$, where $\chi$ is a smooth function with $0 \leq \chi \leq 1$, $\chi \equiv 1$ on $[y_1 - \epsilon, y_2 + \epsilon]$, and $\supp \chi \subset [y_1 - \epsilon - \delta, y_2 + \epsilon + \delta]$. If we define
$
A_1(y) = a'(y_1) + \int_{y_0}^y {A_2(x) dx},
$
where $y_1 \in Range(y_s)$, then by making $\delta$ sufficiently small we get that $A_1$ is uniformly positive.

Finally,
$
A(y) = a(y_0) + \int_{y_0}^y {A_1(x) dx}
$
has all the desired properties.
\end{proof}

Consider the equation obtained by replacing $a$ with $A$ in (\ref{dif}):
\begin{equation}
\begin{array}{l}
y_t - \Delta A(y) = mu + f \text{ on } Q\\
y \mid_{\Sigma} = 0,\ y \mid_{t=0} = y_0.
\end{array}
\label{dif'}
\end{equation}
By our previous considerations, equation (\ref{dif'}) is locally controllable. Writing it~as
\begin{equation}
(y-y_s)_t - \div(A'(y) \dl (y-y_s)) = mu + \div((A'(y)-A'(y_s))\dl y_s),
\end{equation}
one obtains an $L^{\infty}$ estimate on $Q$ for the solution $y$. Indeed, one can first assume as in the proof of Proposition \ref{prop1} that $y \in C^{2+\alpha, 1+\alpha/2}(\Omega)$. After multiplying the equation by $\max((y-y_s - M - tN), 0)$ for suitable $M$ and $N$ and integrating on $[t_1, t_2] \subset [0, T]$, one obtains with the help of (\ref{control}) that
\begin{equation}
\|y-y_s\|_{L^{\infty}([t_1, t_2] \times \Omega)} \leq \|y(t_1) - y_s\|_{\infty} + C(t_1-t_2)(\|y_0-y_s\|_2^{2/q} + \|y-y_s\|_{L^{\infty}([t_1, t_2] \times \Omega)}).
\end{equation}
By subdividing the interval $[0, T]$ into smaller subintervals, one infers that
\begin{equation}
\|y-y_s\|_{L^{\infty}(Q)} \leq C(\|y_0-y_s\|_{\infty} + \|y_0-y_s\|_2^{2/q}).
\end{equation}
By making $\|y_0-y_s\|_{\infty}$ sufficiently small, we obtain that $\|y - y_s\|_{L^{\infty}(Q)} \leq \epsilon$; then $y(Q) \subset [y_1 - \epsilon, y_2 + \epsilon]$.

But on this interval $A$ and $a$ are identical, so it turns out that $y$ is also a solution of the original equation (\ref{dif}) and satisfies the estimates given by Proposition \ref{w2i}, as well as the new $L^{\infty}$ estimate (\ref{estinfinit}).
\end{proof}

\end{document}